\documentclass[10pt,a4paper,english]{amsart}
\usepackage{babel}
\usepackage[numbers]{natbib}
\usepackage{amsfonts, amsmath, wasysym}
\usepackage{amssymb, amsthm}
\usepackage{mathrsfs}
\usepackage{verbatim}
\newfont{\ffont}{cmr10}

\allowdisplaybreaks

\newcommand{\cA}{\mathcal{A}}
\newcommand{\cB}{\mathcal{B}}

\newcommand{\cF}{\mathcal{F}}

\newcommand{\cX}{\mathcal{X}}

\newcommand{\cT}{\mathcal{T}}

\newcommand{\cC}{\mathcal{C}}

\newcommand{\Tr}{\mathrm{Tr}}
\newcommand{\al}{\alpha}
\newcommand{\be}{\beta}
\newcommand{\om}{\omega}
\newcommand{\ga}{\gamma}
\newcommand{\del}{\delta}
\newcommand{\ka}{\kappa}
\newcommand{\Om}{\Omega}
\newcommand{\si}{\sigma}

\newcommand{\Del}{\Delta}
\newcommand{\DelO}{\bar{\Delta}}
\newcommand{\R}{\mathbb{R}}
\newcommand{\C}{\mathbb{C}}
\newcommand{\N}{\mathbb{N}}

\newcommand{\E}{\mathbb{E}}

\newcommand{\bP}{\mathbb{P}}

\newcommand{\eps}{\varepsilon}
\newcommand{\ti}{\times}

\newcommand{\argmin}{\mathrm{argmin}}

\newcommand{\xspace}{\hbox{\kern-2.5pt}}

\renewcommand{\labelenumi}{(\alph{enumi})}

\begin{document}

\newtheorem{definition}{Definition}[section]
\newtheorem{theorem}[definition]{Theorem}
\newtheorem{conjecture}[definition]{Conjecture}
\newtheorem{proposition}[definition]{Proposition}
\newtheorem{corollary}[definition]{Corollary}
\newtheorem{remark}[definition]{Remark}
\newtheorem{lemma}[definition]{Lemma}
\newtheorem{example}{Example}[section]
\newtheorem{exercise}{Exercise}[section]

\title{Tail bounds via generic chaining}
\author{Sjoerd Dirksen}

\address{Universit\"{a}t Bonn\\
Hausdorff Center for Mathematics\\
Endenicher Allee 60\\
53115 Bonn\\
Germany} \email{sjoerd.dirksen@hcm.uni-bonn.de}

\thanks{This research was supported by SFB 1060 of the German Research Foundation (DFG)}
\maketitle

\begin{abstract}
We modify Talagrand's generic chaining method to obtain upper bounds for all $p$-th moments of the supremum of a stochastic process. These bounds lead to an estimate for the upper tail of the supremum with optimal deviation parameters. We apply our procedure to improve and extend some known deviation inequalities for suprema of unbounded empirical processes and chaos processes. As an application we give a significantly simplified proof of the restricted isometry property of the subsampled discrete Fourier transform.
\end{abstract}

\section{Introduction}

This paper is concerned with generic chaining, a method introduced by Talagrand to estimate the expected value of the supremum of a stochastic process. This method grew out of the classical chaining method and the later majorizing measures method, which were developed by, among others, Kolmogorov, Dudley, Fernique and Talagrand, to understand the continuity properties of stochastic processes. Generic chaining yields estimates for the expected value of the supremum in terms of so-called $\ga$-functionals, which measure the metric complexity of the index set of the process \cite{Tal01,Tal05}. The resulting bounds are known to be sharp in several interesting situations. For instance, the famous majorizing measures theorem \cite{Tal87} states that this is the case for suprema of Gaussian processes, provided that the index set is equipped with the canonical metric induced by the process.\par
In practical applications of generic chaining in statistics, compressed sensing, and geometric functional analysis, see e.g.\ \cite{FoR13,Kol11}, it is often not sufficient to have an upper bound for the expected supremum of a process. One also needs to know how probable it is that the supremum of the process exceeds the upper bound. To that aim, a generic chaining bound is typically supplemented with a tail bound for the deviation of the supremum with respect to its expected value. There is an extensive and rapidly growing literature on such deviation inequalities, see for instance the monographs \cite{BLM13,Led01} for a detailed introduction and a historical overview.\par
The purpose of this paper is to provide an alternative to the two-step procedure sketched above. We present a simple and general way to directly obtain upper deviation inequalities for the supremum of a stochastic process $(X_t)_{t\in T}$ using generic chaining. The idea is to alter the generic chaining procedure to produce bounds for not only the first, but for all $p$-th moments of $\sup_{t\in T}|X_t|$. Together with a standard optimization argument using Markov's inequality this yields an upper tail bound. The deviation parameters in the resulting tail bound are sharp up to numerical constants. In particular, the bound is qualitatively as good as the one obtained by combining the usual generic chaining bound for $\E\sup_{t\in T} |X_t|$ with the best possible upper tail bound for the deviation $\sup_{t\in T}|X_t| - \E\sup_{t\in T} |X_t|$.\par
To give a concrete illustration of these statements, consider the simplest case that $(X_t)_{t\in T}$ is a centered Gaussian process and let $d(s,t)=(\E|X_t-X_s|^2)^{1/2}$ be the canonical metric on $T$. Under this assumption, we prove that for some universal constants $C,D>0$ and any $1\leq p<\infty$,
$$\Big(\E\sup_{t\in T}|X_t|^p\Big)^{1/p} \leq C\ga_{2,p}(T,d) + D\si\sqrt{p},$$
where $\ga_{2,p}$ is a truncated version of the $\ga_2$-functional familiar from generic chaining and $\si^2 = \sup_{t\in T} \E(X_t^2)$ is the weak variance of the process. Estimates for the constants $C$ and $D$ are provided in Remark~\ref{rem:singleMet}, although these can certainly be improved. As a direct consequence of the stated $L^p$-bounds we find
\begin{equation}
\label{eqn:tailBdGausInt}
\bP\Big(\sup_{t\in T} |X_t| \geq \sqrt{e}(C\ga_2(T,d) + uD\si)\Big) \leq e^{-u^2/2} \qquad (u\geq 1).
\end{equation}
The bound (\ref{eqn:tailBdGausInt}) matches, up to a possibly worse constant $D$, the upper tail bound obtained by combining the optimal generic chaining estimate for $\E\sup_{t\in T} |X_t|$ with the sharp concentration inequality for suprema of Gaussian processes due to Ibragimov, Sudakov, and Tsirelson \cite[Theorem 5.8]{BLM13}.\par
A first advantage of the method proposed here is its simplicity: an upper tail bound is obtained essentially for free once one uses generic chaining to estimate the expected value. In contrast, the usual proofs of deviation inequalities for suprema of stochastic processes rely on sophisticated tools such as the entropy method, see for example \cite[Chapters 6 and 12]{BLM13}. A second advantage is that the method only requires knowledge of the tail behavior of the individual increments of the process. In particular, one can obtain deviation inequalities for processes with dependent increments, see for example the uniform Azuma-Hoeffding inequality in Corollary~\ref{cor:Azuma-Hoeffding}. In the context of empirical processes, the method can readily cover situations in which the summands of the empirical process are unbounded and/or dependent. Under these conditions deviation inequalities are still scarcely available, see \cite{Ada08,GeL12} for notable exceptions.\par
To demonstrate the wide applicability of our method, we establish an upper tail bound for suprema of stochastic processes in several interesting situations. In Section~\ref{sec:singleMet} we consider two `standard' generic chaining situations. In Theorem~\ref{thm:chainsplitGeneral} we investigate processes which have exponentially decaying increments with respect to a single metric. In Theorem~\ref{thm:mixedTailTal}, we consider processes with a mixed subgaussian-subexponential tail, in particular suprema of empirical processes. The latter result positively answers an open question raised in Talagrand's new book \cite{Tal14}. In the second part of the paper, i.e., Sections~\ref{sec:EmpProc} and \ref{sec:chaos}, we consider two more involved chaining arguments. In Theorem~\ref{thm:supAverages} we find $L^p$-bounds for the supremum of an empirical process which takes the form of an average of squares. It can be viewed as an $L^p$-version of a result due to Mendelson, Pajor and Tomczak-Jaegermann \cite{MPT07}, see Theorem~\ref{thm:MPT} and Corollary~\ref{cor:MPTImproved} for a detailed comparison. In the final section we deal with suprema of second order chaos processes.\par
In Section~\ref{sec:RIP} we use Theorem~\ref{thm:chainsplitGeneral} to simplify the proof of the restricted isometry property of the subsampled discrete Fourier transform. This cornerstone result in compressed sensing was originally discovered by Cand\`{e}s and Tao \cite{CaT06} and was later refined by Rudelson and Vershynin \cite{RuV08} and Rauhut \cite{Rau10}. The argument presented in Section~\ref{sec:RIP} more generally applies to matrices obtained by sampling from bounded orthonormal systems, see Theorem~\ref{thm:BOS}. Let us mention as another application that one can use Theorem~\ref{thm:supAverages} to sharpen the Johnson-Lindenstrauss embedding of Klartag and Mendelson \cite{KlM05}. To keep this paper at a reasonable length we discuss this second application in a separate note \cite{Dir14}.\par
We conclude this introduction with a brief discussion of related work. In \cite{ViV07}, Viens and Vizcarra modified a classical, that is, non-generic chaining argument to obtain upper deviation inequalities for so-called sub-$n$-th chaos processes. Theorem~\ref{thm:chainsplitGeneral} below yields an improvement of this result as a special case. We also improve a deviation inequality for suprema of unbounded empirical processes obtained recently by Van de Geer and Lederer \cite{GeL12}, see the discussion after Corollary~\ref{cor:supEmpProc}. In \cite{Lat11}, Lata{\l}a used a procedure related to ours to prove a comparison result for the strong and weak moments of certain log-concave random vectors. Finally, Krahmer, Mendelson, and Rauhut \cite{KMR13} used a chaining argument to prove $L^p$-bounds for the supremum of a second order chaos process. In Theorem~\ref{thm:chaos} we give an improvement of their bounds with a simplified proof.

\section{Preliminaries}

Throughout, we will use $(\Om,\cF,\bP)$ to denote a probability space and write $\E$ for the expected value. To describe the tail behavior of random variables we consider for every $0<\al<\infty$ the function
$$\psi_{\al}(x) = \exp(x^{\al}) - 1 \qquad (x\geq 0).$$
For a complex-valued random variable $X$ we define
$$\|X\|_{\psi_{\al}} = \inf\{C>0 \ : \ \E\psi_{\al}(|X|/C)\leq 1\}.$$
If $\|X\|_{\psi_{\al}}<\infty$ then we call $X$ a \emph{$\psi_{\al}$-random variable}. It is common to say that $X$ is \emph{subgaussian} if $\|X\|_{\psi_2}<\infty$ and \emph{subexponential} if $\|X\|_{\psi_1}<\infty$. If $\al\geq 1$ then $\psi_{\al}$ is an Orlicz function and the space
$$L_{\psi_{\al}}(\Om,\cF,\bP) = \{X:\Om\rightarrow\C \ \mathrm{measurable} \ : \ \|X\|_{\psi_{\al}}<\infty\}$$
is an Orlicz space. For $0<\al<1$ the space $L_{\psi_{\al}}$ is only a quasi-Banach space. We will make use of the following H\"{o}lder type inequality, which can be derived from Young's inequality: if $X,Y$ are $\psi_2$-random variables, then $XY$ is $\psi_1$ and
\begin{equation}
\label{eqn:C-SPsi}
\|XY\|_{\psi_1}\leq \|X\|_{\psi_2} \|Y\|_{\psi_2}.
\end{equation}
For more information on Orlicz spaces we refer to \cite{KrR61}.\par
Let us recall some familiar concepts from generic chaining \cite{Tal05}. Let $\cX$ be a normed linear space and let $(T,d)$ be a semi-metric space, i.e., $d(x,z) \leq d(x,y) + d(y,z)$ and $d(x,y)=d(y,x)$ for $x,y,z \in T$. To avoid complications with the measurability of suprema of stochastic processes we will always assume that the cardinality $|T|$ of $T$ is \emph{finite}. Criteria for measurability of the supremum of a stochastic process in the case of an (uncountably) infinite index set can be found in \cite[Section 1.7]{VaW96}. We denote the \emph{diameter} of $T$ with respect to $d$ by
$$\Del_d(T) = \sup_{s,t \in T} d(s,t).$$
We say that an $\cX$-valued process $(X_t)_{t\in T}$ is \emph{$\psi_{\al}$ with respect to $d$} if for all $s,t\in T$,
\begin{equation}
\label{eqn:defPsiAlProcess}
\bP(\|X_t - X_s\|\geq ud(t,s)) \leq 2\exp(-u^{\al}) \qquad (u\geq 0).
\end{equation}
A sequence $\cT=(T_n)_{n\geq 0}$ of subsets of $T$ is called \emph{admissible} if $|T_0|=1$ and $|T_n|\leq 2^{2^n}$ for all $n\geq 1$. For any $0<\al<\infty$, the \emph{$\ga_{\al}$-functional} of $(T,d)$ is defined by
$$\ga_{\al}(T,d) = \inf_{\cT}\sup_{t\in T} \sum_{n=0}^{\infty} 2^{n/{\al}}d(t,T_n),$$
where the infimum is taken over all admissible sequences and we write $d(t,T_n)=\inf_{s\in T_n}d(t,s)$.\par
For any given $u>0$ let $N(T,d,u)$ denote the covering number of $T$, i.e., the smallest number of balls of radius $u$ in $(T,d)$ needed to cover $T$. One can always estimate
\begin{equation}
\label{eqn:gammaFunEstEntInt}
\ga_{\al}(T,d) \lesssim_{\al} \int_0^{\infty} \Big(\log N(T,d,u)\Big)^{1/\al} \ du,
\end{equation}
see \cite[Section 1.2]{Tal05} for the case $\al=2$ (the other cases are similar). However, the reverse estimate fails in general \cite[Section 2.1]{Tal05}.\par
We conclude by fixing some notation. We use $\|\cdot\|_p$, $1\leq p\leq\infty$, to denote the $\ell^p$-norms. We will write $A\lesssim_{\be} B$ if $A\leq C_{\be}B$ for a constant $C_{\be}$ which only depends on a parameter $\be$. Finally, if $S$ is a finite set and $\pi:S\rightarrow\R_+$ is a map, then $\argmin_{s\in S} \pi(s)$ denotes a minimizer of this map, which may not be unique.

\section{Suprema of $\psi_{\al}$ and mixed tail processes}
\label{sec:singleMet}

We begin our discussion by considering two standard generic chaining situations. First, in Theorem~\ref{thm:chainsplitGeneral} we establish tail bounds for suprema of $\psi_{\al}$ processes. At the end of the section, in Theorem~\ref{thm:mixedTailTal}, we do the same for processes with a mixed tail.\par
In the formulation of our $L^p$-bounds we will make use of the following truncated version of the $\gamma$-functionals. For a given $1\leq p<\infty$, we will always write $l:=\lfloor\log_2(p)\rfloor$, where $\lfloor\cdot\rfloor$ denotes the integer part. We define
\begin{equation}
\label{eqn:gammaTruncated}
\ga_{\al,p}(T,d) = \inf_{\cT}\sup_{t\in T} \sum_{n\geq l} 2^{n/{\al}}d(t,T_n).
\end{equation}
Clearly, $\ga_{\al,p}(T,d)\leq \ga_{\al}(T,d)$ for all $1\leq p<\infty$ and $\ga_{\al,1}(T,d)=\ga_{\al}(T,d)$. Since we assume $T$ to be finite, the infimum in (\ref{eqn:gammaTruncated}) is actually attained. We will call a sequence $\cT$ that achieves the infimum \emph{optimal}.
\begin{remark}
\emph{If $T$ is an infinite set, then the $L^p$-bounds presented below continue to hold if we interpret $\E\sup_{t\in T}\|X_t-X_{t_0}\|^p$ as the lattice supremum}
$$\E\sup_{t\in T}\|X_t-X_{t_0}\|^p := \sup\{\E\sup_{t\in F}\|X_t-X_{t_0}\|^p \ : \ F\subset T, \ |F|<\infty\}.$$
\end{remark}
\begin{theorem}
\label{thm:chainsplitGeneral}
Let $0<\al<\infty$. If $(X_t)_{t\in T}$ is $\psi_{\al}$, then there exist constants $C_{\al},D_{\al}>0$ depending only on $\al$, such that for any $t_0 \in T$ and $1\leq p<\infty$,
\begin{equation}
\label{eqn:chainsplitGeneralLp}
\Big(\E\sup_{t\in T}\|X_t - X_{t_0}\|^p\Big)^{1/p} \leq C_{\al}\ga_{\al,p}(T,d) + 2 \sup_{t\in T} (\E\|X_t-X_{t_0}\|^p)^{1/p}.
\end{equation}
As a consequence, for any $u\geq 1$,
\begin{equation}
\label{eqn:chainsplitGeneralTail}
\bP\Big(\sup_{t\in T}\|X_t - X_{t_0}\| \geq e^{1/\al}(C_{\al}\ga_{\al}(T,d) + uD_{\al}\Del_d(T))\Big) \leq \exp(-u^{\al}/\al).
\end{equation}
\end{theorem}
\begin{remark}
\label{rem:singleMet}
\begin{enumerate}
\renewcommand{\labelenumi}{(\roman{enumi})}
\item \emph{If $(X_t)_{t\in T}$ is a real-valued Gaussian process and $d$ is the canonical distance $d(s,t) := (\E|X_t-X_s|^2)^{1/2}$, then Theorem~\ref{thm:chainsplitGeneral} produces a sharp $L^p$-bound (up to universal constants). Indeed, Talagrand's majorizing measures theorem \cite{Tal87,Tal05} states that
$$\ga_2(T,d)\lesssim \E\sup_{t\in T} |X_t - X_{t_0}|.$$
Moreover, it is of course always true that
$$\sup_{t\in T} (\E|X_t-X_{t_0}|^p)^{1/p} \leq \big(\E\sup_{t\in T}|X_t - X_{t_0}|^p\big)^{1/p}.$$
\item Although deviation inequalities are not discussed in \cite{Tal05}, it is implicitly used there that
$$\bP\Big(\sup_{t\in T}\|X_t - X_{t_0}\|\geq uc_1\ga_2(T,d)\Big) \leq c_2e^{-u^2/2} \qquad (u\geq 2)$$
in the case $\al=2$. Since $\Del(T)\leq \ga_2(T,d)$ (and $\Del(T)$ is potentially much smaller), this bound is qualitatively worse than (\ref{eqn:chainsplitGeneralTail}). After the first version of this paper was finished, the author learned that (\ref{eqn:chainsplitGeneralTail}) is proved for $\al=2$ in Talagrand's new book using a different method (see \cite[Theorem 2.2.27]{Tal14}).
\item Let $n\in \N$. In \cite[Theorem 3.1]{ViV07}, it was shown that if $(X_t)_{t\in T}$ is a} sub-$n$-th chaos process\emph{, meaning in the terminology used here that it is $\psi_{2/n}$, then it satisfies the tail bound
    \begin{equation}
    \label{eqn:ViV}
    \bP\Big(\sup_{t\in T}|X_t-X_{t_0}|\geq C_n M_n + uC_n'\Del_d(T)\Big) \leq 2\exp\Big(-\frac{u^{2/n}}{2}\Big),
    \end{equation}
    where $C_n,C_n'$ are constants depending only on $n$ and $M_n$ is the entropy integral
    $$M_n = \int_0^{\infty} \Big(\log N(T,d,u)\Big)^{n/2} \ du.$$
    By (\ref{eqn:gammaFunEstEntInt}) this result is a direct consequence of Theorem~\ref{thm:chainsplitGeneral} (with possibly different constants) but not vice versa. Note that already in the subgaussian case $n=1$ the bound (\ref{eqn:ViV}) is not sharp (see \cite[Section 2.1]{Tal05}).
\item To keep our exposition clear we do not keep precise track of the numerical constants in the chaining arguments. However, from the proof below it is clear that $C_{\al},D_{\al}$ are decreasing in $\al$. Moreover, one can decrease $C_{\al}$ at the expense of increasing $D_{\al}$ (and vice versa). To give an idea of their order of magnitude, one can readily deduce from the proof (without making any effort to optimize the constants) that $C_2 \leq (1+\sqrt{2})(16\sqrt{\pi}e^{1/2e}+\sqrt{2})\leq 86$ and $D_2\leq 4\sqrt{2\pi}e^{1/e}e^{-1/2}\leq 9$. Although these estimates can certainly be improved, the method cannot yield optimal numerical constants. In particular losses occur when passing between moment and tail bounds (cf.\ Lemmas~\ref{lem:MomentsToTails} and \ref{lem:TailsToMoments}).}
\end{enumerate}
\end{remark}
\begin{proof}[Proof of Theorem~\ref{thm:chainsplitGeneral}]
Let $\cT=(T_n)_{n\geq 0}$ be an optimal admissible sequence for $\ga_{\al,p}(T,d)$ and let $\pi=(\pi_n)_{n\geq 0}$ be a sequence of functions $\pi_n:T\rightarrow T_n$ defined by $\pi_n(t)=\argmin_{s\in T_n}d(s,t)$. Set $l=\lfloor\log_2(p)\rfloor$. We make the decomposition
\begin{equation}
\label{eqn:triangleSplit}
\Big(\E\sup_{t\in T}\|X_t - X_{t_0}\|^p\Big)^{\frac{1}{p}} \leq \Big(\E\sup_{t\in T}\|X_t - X_{\pi_l(t)}\|^p\Big)^{\frac{1}{p}} + \Big(\E\sup_{t\in T}\|X_{\pi_l(t)} - X_{t_0}\|^p\Big)^{\frac{1}{p}}.
\end{equation}
We estimate the second term on the right hand side by Lemma~\ref{lem:supSmallSet},
\begin{align}
\label{eqn:boundSecondSplit}
\Big(\E\sup_{t\in T}\|X_{\pi_l(t)} - X_{t_0}\|^p\Big)^{\frac{1}{p}} & \leq 2\sup_{t\in T}(\E\|X_{\pi_l(t)} - X_{t_0}\|^p)^{\frac{1}{p}} \nonumber\\
& \leq 2\sup_{t\in T}(\E\|X_{t} - X_{t_0}\|^p)^{\frac{1}{p}}.
\end{align}
For the first term, we write the telescoping sum
$$X_t - X_{\pi_l(t)} = \sum_{n>l} X_{\pi_n(t)} - X_{\pi_{n-1}(t)}.$$
Since the increments of $X$ are $\psi_{\al}$, we have for $n>l$,
\begin{equation*}
\bP\Big(\|X_{\pi_n(t)} - X_{\pi_{n-1}(t)}\|\geq u2^{n/\al}d(\pi_n(t),\pi_{n-1}(t))\Big) \leq 2\exp(-u^{\al} 2^n).
\end{equation*}
Note that $|\{(\pi_n(t),\pi_{n-1}(t)); t\in T\}|\leq |T_n| \ |T_{n-1}| \leq 2^{2^n}2^{2^{n-1}}\leq 2^{2^{n+1}}$. Therefore, if $\Om_{u,p}$ denotes the event
$$\forall n>l, \forall t\in T \ : \ \|X_{\pi_n(t)} - X_{\pi_{n-1}(t)}\|\leq u2^{n/\al}d(\pi_n(t),\pi_{n-1}(t)),$$
then Lemma~\ref{lem:unionBoundEst} shows that
$$\bP(\Om_{u,p}^c) \leq c \exp(-p u^{\al}/4) \qquad (u\geq 2^{1/\al}).$$
If the event $\Om_{u,p}$ occurs, then
\begin{align*}
\Big\|\sum_{n>l} X_{\pi_n(t)} - X_{\pi_{n-1}(t)}\Big\| & \leq \sum_{n>l} \|X_{\pi_n(t)} - X_{\pi_{n-1}(t)}\| \\
& \leq u \sum_{n>l} 2^{n/\al} d(\pi_n(t),\pi_{n-1}(t)) \leq (1+2^{1/\al})\ga_{\al,p}(T,d).
\end{align*}
Thus, $\sup_{t\in T} \|X_t-X_{\pi_l(t)}\|\leq u(1+2^{1/\al})\ga_{\al,p}(T,d)$. In conclusion,
$$\bP\Big(\sup_{t\in T} \|X_t-X_{\pi_l(t)}\|> u(1+2^{1/\al})\ga_{\al,p}(T,d)\Big) \leq c\exp(-p u^{\al}/4),$$
whenever $u\geq 2^{1/\al}$. Lemma~\ref{lem:LpBdElem} implies that
\begin{equation}
\label{eqn:boundFirstSplit}
\Big(\E\sup_{t\in T}\|X_t - X_{\pi_l(t)}\|^p\Big)^{1/p} \leq C_{\al}\ga_{\al,p}(T,d).
\end{equation}
The moment bound (\ref{eqn:chainsplitGeneralLp}) follows by combining (\ref{eqn:triangleSplit}), (\ref{eqn:boundSecondSplit}) and (\ref{eqn:boundFirstSplit}). For the tail bound, note that (\ref{eqn:defPsiAlProcess}) and Lemma~\ref{lem:TailsToMoments} together imply that
$$\sup_{t\in T} (\E\|X_t - X_{t_0}\|^p)^{1/p} \leq D_{\al}\Del_d(T)p^{1/\al}.$$
The final assertion follows by using this estimate in (\ref{eqn:chainsplitGeneralLp}) and applying Lemma~\ref{lem:MomentsToTails}.
\end{proof}
Note that Theorem~\ref{thm:chainsplitGeneral} does not require any independence assumptions on the increments of the process $(X_t)_{t\in T}$. To illustrate this, we recall the Azuma-Hoeffding inequality, see e.g.\ \cite[Lemma 4.1]{Led01}. If $X=(X_k)_{0\leq k\leq n}$ is a discrete-time real-valued martingale and $\Delta_k(X) = X_k-X_{k-1}$ denotes its $k$-th difference, then
$$\bP(|X_n-X_0|\geq u) \leq 2\exp\Big(-\frac{u^2}{2\sum_{k=1}^n\|\Delta_k(X)\|_{\infty}^2}\Big) \qquad (u\geq 0).$$
Combined with Theorem~\ref{thm:chainsplitGeneral} we immediately obtain the following uniform version of the Azuma-Hoeffding bound.
\begin{corollary}
\label{cor:Azuma-Hoeffding}
Let $X_t=(X_{t,k})_{1\leq k\leq n}$, $t\in T$, be a family of discrete-time martingales with respect to the same filtration. We consider the metric
$$d(s,t) = \Big(\sum_{k=1}^n\|\Delta_k(X_t - X_s)\|_{\infty}^2\Big)^{1/2}.$$
For any $u\geq 1$,
$$\bP\Big(\sup_{t\in T} |X_{t,n} - X_{t,0}| \geq \sqrt{e}(C_2\ga_2(T,d) + D_2\Del_d(T)u)\Big) \leq e^{-u^2/2}.$$
\end{corollary}
Let $d_{1},d_{2}$ be two semi-metrics on $T$. We say that a process $(X_t)_{t\in T}$ has mixed subgaussian-subexponential increments, or simply \emph{has a mixed tail}, with respect to the pair $(d_{1},d_{2})$ if for all $s,t\in T$,
\begin{equation}
\label{eqn:mixedTailDef}
\bP(\|X_t - X_s\|\geq \sqrt{u}d_{2}(t,s) + ud_{1}(t,s))  \leq 2e^{-u} \qquad (u\geq 0).
\end{equation}
This means that the first part of the tail behaves as the tail of a subgaussian random variable and the second part as the tail of a subexponential random variable. In Theorem~\ref{thm:mixedTailTal} we prove a tail bound for the supremum of a process with a mixed tail. The result improves \cite[Theorem 1.2.9]{Tal05} and, in fact, positively answers an open question in Talagrand's new book (see the discussion after \cite[Theorem 2.2.28]{Tal14}).\par
In the proof it will be convenient to work with an alternative definition of the $\ga_{\al}$-functionals. Let us say that a sequence $\cA=(\cA_n)_{n\geq 0}$ of partitions of $T$ is \emph{admissible} if it is increasing with respect to the refinement ordering and $|\cA_n|\leq 2^{2^n}$. For any $t\in T$, let $A_n(t)$ be the unique element in the partition $\cA_n$ containing $t$. We now set
$$\ga_{\al}'(T,d) = \inf_{\cA}\sup_{t\in T} \sum_{n=0}^{\infty} 2^{n/\al}\Del_d(A_n(t)),$$
where the infimum is taken over all admissible partitions $\cA$ of $T$. It can be shown that
\begin{equation}
\label{eqn:gaEquivGaPrime}
\ga_{\al}(T,d) \leq \ga_{\al}'(T,d) \lesssim_{\al} \ga_{\al}(T,d),
\end{equation}
see the discussion following \cite[Theorem 1.3.5]{Tal05} for details.
\begin{theorem}
\label{thm:mixedTailTal}
If $(X_t)_{t\in T}$ has a mixed tail, then there is a constant $C>0$ such that for any $1\leq p<\infty$,
\begin{equation}
\label{eqn:mixedTailTal}
\Big(\E \sup_{t\in T} \|X_t-X_{t_0}\|^p\Big)^{1/p} \leq C(\ga_2(T,d_2) + \ga_1(T,d_1)) + 2\sup_{t\in T} (\E\|X_t-X_{t_0}\|^p)^{1/p}.
\end{equation}
As a consequence, there are constants $c,C>0$ such that for any $u\geq 1$,
\begin{equation*}
\bP\Big(\sup_{t\in T} \|X_t-X_{t_0}\| \geq C(\ga_2(T,d_2) + \ga_1(T,d_1)) + c(\sqrt{u}\Del_{d_2}(T) + u\Del_{d_1}(T))\Big)\leq e^{-u}.
\end{equation*}
\end{theorem}
\begin{proof}
We select two admissible sequences of partitions $\cB=(\cB_n)_{n\geq 0}$ and $\cC=(\cC_n)_{n\geq 0}$ such that
\begin{align*}
& \sup_{t\in T} \sum_{n\geq 0} 2^{n/2}\Del_{d_2}(B_n(t)) \leq 2\ga_2'(T,d_2) \\
& \sup_{t\in T} \sum_{n\geq 0} 2^{n}\Del_{d_1}(C_n(t)) \leq 2\ga_1'(T,d_1).
\end{align*}
Let $\cA_n$ be the partition generated by $\cB_{n-1}$ and $\cC_{n-1}$, i.e.,
$$\cA_n = \{B\cap C \ : \ B \in \cB_{n-1}, \ C \in \cC_{n-1}\}.$$
Then $\cA=(\cA_n)_{n\geq 0}$ is increasing and
$$|\cA_n| \leq |\cB_{n-1}| \ |\cC_{n-1}| \leq 2^{2^{n-1}}2^{2^{n-1}} = 2^{2^n},$$
so $\cA$ is admissible. Observe that $A_n(t)=B_{n-1}(t)\cap C_{n-1}(t)$. For every $n\geq 0$ define a subset $T_n$ of $T$ by selecting exactly one point from each $A \in \cA_n$. In this way, we obtain an admissible sequence $\cT=(T_n)_{n\geq 0}$ of subsets of $T$. For every $n\in \N_{\geq 0}$ and $t\in T$ we let $\pi_n(t)$ be the unique element of $T_n\cap A_n(t)$. This yields a sequence $\pi=(\pi_n)_{n\geq 0}$ of maps $\pi_n:T\rightarrow T_n$.\par
Set $l=\lfloor\log_2(p)\rfloor$. As in the proof of Theorem~\ref{thm:chainsplitGeneral}, we make the decomposition (\ref{eqn:triangleSplit}) and estimate the second term on the right hand side of (\ref{eqn:triangleSplit}) as in (\ref{eqn:boundSecondSplit}). For the first term, we write the telescoping sum
$$X_t - X_{\pi_l(t)} = \sum_{n>l} X_{\pi_n(t)} - X_{\pi_{n-1}(t)}.$$
Since $X$ has a mixed tail, we have for $n>l$ and $u\geq 0$,
\begin{align*}
& \bP\Big(\|X_{\pi_n(t)} - X_{\pi_{n-1}(t)}\|\geq \sqrt{u}2^{n/2}d_{2}(\pi_n(t),\pi_{n-1}(t)) + u2^n d_{1}(\pi_n(t),\pi_{n-1}(t))\Big) \\
& \qquad \qquad \qquad \qquad \qquad \qquad \qquad \qquad \leq 2\exp(-u2^n).
\end{align*}
Note that $|\{(\pi_n(t),\pi_{n-1}(t)); t\in T\}|\leq |T_n| \ |T_{n-1}| \leq 2^{2^n}2^{2^{n-1}}\leq 2^{2^{n+1}}$. Therefore, if $\Om_{u,p}$ denotes the event
\begin{align*}
& \forall n>l, \forall t\in T \ : \ \|X_{\pi_n(t)} - X_{\pi_{n-1}(t)}\| \\
& \qquad \qquad \qquad \qquad \qquad \geq \sqrt{u}2^{n/2}d_{2}(\pi_n(t),\pi_{n-1}(t)) + u2^n d_{1}(\pi_n(t),\pi_{n-1}(t)),
\end{align*}
then Lemma~\ref{lem:unionBoundEst} shows that
$$\bP(\Om_{u,p}^c) \leq c \exp(-p u/4) \qquad (u\geq 2).$$
If the event $\Om_{u,p}$ occurs, then
\begin{align*}
& \Big\|\sum_{n>l} X_{\pi_n(t)} - X_{\pi_{n-1}(t)}\Big\| \\
& \qquad \leq \sum_{n>l} \|X_{\pi_n(t)} - X_{\pi_{n-1}(t)}\| \\
& \qquad \leq \sqrt{u}\sum_{n>l}2^{n/2}d_{2}(\pi_n(t),\pi_{n-1}(t)) + u \sum_{n>l}2^n d_{1}(\pi_n(t),\pi_{n-1}(t)).
\end{align*}
Observe that for $n\geq 2$ we have $\pi_n(t),\pi_{n-1}(t) \in A_{n-1}(t)\subset B_{n-2}(t)$ and so
$$d_2(\pi_n(t),\pi_{n-1}(t)) \leq \Del_{d_2}(B_{n-2}(t)).$$
Also,
$$d_2(\pi_1(t),\pi_0(t)) \leq \Del_{d_2}(B_0(t)) = \Del_{d_2}(T).$$
Therefore, by our choice of $\cB$,
\begin{align*}
\sum_{n>l} 2^{n/2}d_2(\pi_n(t),\pi_{n-1}(t)) & \leq \sum_{n>l} 2^{n/2}\Del_{d_2}(B_{n-2}(t)) \\
& \leq 4 \sum_{n\geq 0} 2^{n/2}\Del_{d_2}(B_{n}(t)) \leq 8 \ga_2'(T,d_2).
\end{align*}
Analogously, by our choice of $\cC$,
$$\sum_{n>l} 2^{n}d_1(\pi_n(t),\pi_{n-1}(t)) \leq 12\ga_1'(T,d_1)$$
Thus, $\sup_{t\in T} \|X_t-X_{\pi_l(t)}\|\leq 8\sqrt{u}\ga_2'(T,d_2) + 12u\ga_1'(T,d_1)$. As a consequence, we conclude that
$$\bP\Big(\sup_{t\in T} \|X_t-X_{\pi_l(t)}\|> 12u(\ga_2'(T,d_2) + \ga_1'(T,d_1))\Big) \leq c\exp(-p u/4),$$
whenever $u\geq 2$. By Lemma~\ref{lem:LpBdElem} and (\ref{eqn:gaEquivGaPrime})
\begin{equation}
\label{eqn:boundFirstSplit}
\Big(\E\sup_{t\in T}\|X_t - X_{\pi_l(t)}\|^p\Big)^{1/p} \lesssim \ga_2'(T,d_2) + \ga_1'(T,d_1) \lesssim \ga_2(T,d_2) + \ga_1(T,d_1).
\end{equation}
This proves the moment bound (\ref{eqn:mixedTailTal}). For the tail bound, note that (\ref{eqn:mixedTailDef}) and Lemma~\ref{lem:TailsToMoments} together imply that
$$\sup_{t\in T} (\E\|X_t - X_{t_0}\|^p)^{1/p} \lesssim \Del_{d_2}(T)\sqrt{p} + \Del_{d_1}(T)p.$$
The assertion follows by using this estimate in (\ref{eqn:mixedTailTal}) and applying Lemma~\ref{lem:MomentsToTails}.
\end{proof}
In Section~\ref{sec:EmpProc} below we use Theorem~\ref{thm:mixedTailTal} to derive tail bounds for suprema of empirical processes.

\section{Restricted isometry constants of subsampled unitary matrices}
\label{sec:RIP}

In this section we present an application of Theorem~\ref{thm:chainsplitGeneral} in compressed sensing. We use the following terminology. For a given $s\in \N$, the \emph{$s$-th restricted isometry constant} $\del_s$ of an $m\ti N$ matrix $A$ is the smallest constant $\del\geq 0$ such that
$$(1-\del)\|x\|_2^2 \leq \|Ax\|_2^2 \leq (1+\del)\|x\|_2^2,$$
for all $s$-sparse $x\in \C^N$. Equivalently, if we let $\|x\|_0=|\{i \ : \ x_i\neq 0\}|$ and
$$D_{s,N} = \{x \in \C^N \ : \ \|x\|_2=1, \ \|x\|_0\leq s\},$$
then
$$\del_s = \sup_{x\in D_{s,N}} \Big|\|Ax\|_2^2 - 1\Big|.$$
The restricted isometry constants play an important role in compressed sensing, see \cite[Chapter 6]{FoR13} for more information. We restrict ourselves to the task of giving a simpler proof of the fact that the random matrix obtained by uniformly sampling rows of the discrete Fourier transform has small restricted isometry constants with high probability. This result was obtained by Cand\`{e}s and Tao in the influential paper \cite{CaT06}. An improved result was later found by Rudelson and Vershynin \cite{RuV08} using a different method. Finally, by elaborating on this method a better probability estimate was obtained by Rauhut \cite{Rau10}.\par
We consider the following (more general) setup. Let $U$ be a unitary $N\ti N$ matrix and suppose that for some constant $K\geq 1$,
\begin{equation}
\label{eqn:assUkl}
\sup_{k,l} \sqrt{N} |U_{kl}| \leq K.
\end{equation}
We consider a sequence $(\theta_i)_{1\leq i\leq N}$ of i.i.d.\ copies of the random selector $\theta:\Om\rightarrow \{0,1\}$ which satisfies
$$\bP(\theta = 1) = \frac{m}{N}.$$
Let $I=\{i \in [N] \ : \ \theta_i=1\}$ be the random set of selected indices and note that its expected cardinality is $\E|I|=m$. Let $R_I:\C^N\rightarrow \C^{|I|}$ be the operator which restricts a vector to its entries in $I$ and consider the subsampled and rescaled matrix
\begin{equation}
\label{eqn:defUI}
U_I := \sqrt{\frac{N}{m}}R_I U.
\end{equation}
The subsampled discrete Fourier transform corresponds to taking
$$U_{kl} = \frac{1}{\sqrt{N}} \exp(2\pi i(k-1)(l-1)/N) \qquad (k,l=1,\ldots,N).$$
\begin{theorem}
\label{thm:RIPFourier}
\cite{CaT06, RuV08,Rau10} Let $U$ and $I$ be as above. Set $\del_s = \del_s(U_I)$. There exist universal constants $d_1,d_2>0$ such that for any given $s \in \N$ and $0<\del,\eta<1$, we have $\bP(\del_s \geq \del) \leq \eta$, provided that
\begin{equation}
\label{eqn:conditionFourier}
m\geq sK^2\del^{-2}\max\{d_1\log^2 s \log m\log N, d_2\log(\eta^{-1})\}.
\end{equation}
\end{theorem}
The proof of Theorem~\ref{thm:RIPFourier} in \cite{Rau10} (see also \cite[Theorem 12.32]{FoR13}), which refines the approach in \cite{RuV08}, consists of two parts: firstly, the expected value of $\del_s$ is estimated using a (classical) chaining argument. Secondly, a deviation inequality for suprema of bounded empirical processes is used to show that $\del_s$ is typically not much larger than its expected value. Here we shorten the proof by merging these two steps, hence dispensing with the concentration inequality. Note that we still use a certain entropy bound obtained in \cite{RuV08} (see (\ref{eqn:entBoundRV}) below), which is nontrivial to prove.
\begin{proof}[Proof of Theorem~\ref{thm:RIPFourier}]
Let $U_i$ be the $i$-th row of $U$. For every $x\in D_{s,N}$ we define $f_x(\theta_i)=\theta_i\frac{1}{\sqrt{m}}\langle\sqrt{N}U_i,x\rangle$. Since $U$ is unitary,
\begin{equation}
\label{eqn:applUunitary}
\sum_{i=1}^N \E f_x^2(\theta_i) = \frac{N}{m}\sum_{i=1}^N \E(\theta_i) |\langle U_i,x\rangle|^2 = \sum_{i=1}^N|\langle U_i,x\rangle|^2 = \|U x\|_2^2 = 1,
\end{equation}
and therefore we can write
$$\del_s = \sup_{x\in D_{s,N}} \Big|\sum_{i=1}^N f_x^2(\theta_i) - \E f_x^2(\theta_i)\Big|.$$
Fix $1\leq p<\infty$ and let $(\eps_i)_{i\geq 1}$ be a Rademacher sequence, i.e., a sequence of independent symmetric Bernoulli random variables. By a standard symmetrization argument
\cite[Lemma 6.3]{LeT91},
\begin{equation}
\label{eqn:symmFourier}
(\E\del_s^p)^{1/p} \leq 2\Big(\E\E_{\eps}\sup_{x\in D_{s,N}} \Big|\sum_{i=1}^N \eps_i f_x^2(\theta_i)\Big|^p\Big)^{1/p}.
\end{equation}
Now we fix $\om \in \Om$ and let $t_i=\theta_i(\om)$. By Hoeffding's inequality,
$$\bP_{\eps}\Big(\sum_{i=1}^N \eps_i (f_x^2(t_i) - f_y^2(t_i)) \geq u\Big(\sum_{i=1}^N (f_x^2(t_i) - f_y^2(t_i))^2\Big)^{1/2}\Big)\leq \exp(-u^2/2).$$
Moreover,
\begin{align*}
\Big(\sum_{i=1}^N (f_x^2(t_i) - f_y^2(t_i))^2\Big)^{1/2} & = \Big(\sum_{i=1}^N (f_x(t_i) - f_y(t_i))^2(f_x(t_i) + f_y(t_i))^2\Big)^{1/2} \\
& \leq 2\sup_{z\in D_{s,N}}\Big(\sum_{i=1}^N f_z^2(t_i)\Big)^{1/2} \max_{1\leq i\leq N}|f_x(t_i)-f_y(t_i)|.
\end{align*}
In conclusion, the process
$$x\mapsto \sum_{i=1}^N \eps_i f^2_x(t_i)$$
has subgaussian increments with respect to the metric
$$d(x,y) = \sup_{z\in D_{s,N}}\Big(\sum_{i=1}^N f_z^2(t_i)\Big)^{1/2}d_t(x,y),$$
where $d_t$ denotes the metric
$$d_t(x,y)=\max_{1\leq i\leq N}|f_x(t_i)-f_y(t_i)|.$$
By Theorem~\ref{thm:chainsplitGeneral},
\begin{align*}
& \Big(\E_{\eps}\sup_{x\in D_{s,N}} \Big|\sum_{i=1}^N \eps_i f_x^2(t_i)\Big|^p\Big)^{1/p} \\
& \qquad  \lesssim \ga_2(D_{s,N},d_t)\sup_{x\in D_{s,N}}\Big(\sum_{i=1}^N f_x^2(t_i)\Big)^{1/2} + \sup_{x\in D_{s,N}}\Big(\E_{\eps}\Big|\sum_{i=1}^N \eps_i f_x^2(t_i)\Big|^p\Big)^{1/p}.
\end{align*}
We apply (\ref{eqn:gammaFunEstEntInt}) with $\al=2$, i.e.,
$$\ga_2(D_{s,N},d_t) \lesssim \int_0^{\infty}\Big(\log N(D_{s,N},d_t,u)\Big)^{1/2} \ du$$
and use that Rudelson and Vershynin already proved that (see inequalities (3.8) and (3.9) in \cite{RuV08}, or \cite{FoR13,Rau10})
\begin{equation}
\label{eqn:entBoundRV}
\int_0^{\infty}\Big(\log N(D_{s,N},d_t,u)\Big)^{1/2} \ du \lesssim K\sqrt{\frac{s}{m}}\log s\sqrt{\log m}\sqrt{\log N}.
\end{equation}
Moreover, by Khintchine's (or Hoeffding's) inequality,
$$\Big(\E_{\eps}\Big|\sum_{i=1}^N \eps_i f_x^2(t_i)\Big|^p\Big)^{1/p} \leq \sqrt{p}\Big(\sum_{i=1}^N f_x^4(t_i)\Big)^{1/2} \leq \sqrt{p}\Big(\sum_{i=1}^N f_x^2(t_i)\Big)^{1/2}\max_{1\leq i\leq N}|f_x(t_i)|.$$
By H\"{o}lder's inequality and (\ref{eqn:assUkl})
$$\max_{1\leq i\leq N} |f_x(t_i)| \leq \max_{1\leq i\leq N}\frac{1}{\sqrt{m}}\|\sqrt{N}U_i\|_{\infty}\|x\|_1 \leq K\sqrt{\frac{s}{m}},$$
where the final inequality follows from the $s$-sparsity of $x$. Collecting our estimates we find
\begin{align*}
& \Big(\E_{\eps}\sup_{x\in D_{s,N}} \Big|\sum_{i=1}^N \eps_i f_x^2(t_i)\Big|^p\Big)^{1/p} \\
& \qquad \lesssim \sup_{x\in D_{s,N}}\Big(\sum_{i=1}^N f_x^2(t_i)\Big)^{1/2}K\sqrt{\frac{s}{m}}\Big(\log s\sqrt{\log m}\sqrt{\log N} + \sqrt{p}\Big).
\end{align*}
We now take the $L^p$ norm on $\Om$ on both sides and obtain using (\ref{eqn:symmFourier})
\begin{align}
\label{eqn:QIdeltas}
& \Big(\E\sup_{x\in D_{s,N}} \Big|\sum_{i=1}^N f_x^2(\theta_i)-\E f_x^2(\theta_i)\Big|^p\Big)^{1/p} \nonumber\\
& \ \ \lesssim \Big(\E\sup_{x\in D_{s,N}}\Big(\sum_{i=1}^N f_x^2(\theta_i)\Big)^{p/2}\Big)^{1/p}K\sqrt{\frac{s}{m}}\Big(\log s\sqrt{\log m}\sqrt{\log N} + \sqrt{p}\Big) \nonumber\\
& \ \ \lesssim \Big(\E\sup_{x\in D_{s,N}}\Big|\sum_{i=1}^N f_x^2(\theta_i) - \E f_x^2(\theta_i)\Big|^p\Big)^{1/2p}K\sqrt{\frac{s}{m}}\Big(\log s\sqrt{\log m}\sqrt{\log N} + \sqrt{p}\Big) \nonumber\\
& \ \ \qquad + K\sqrt{\frac{s}{m}}\Big(\log s\sqrt{\log m}\sqrt{\log N} + \sqrt{p}\Big),
\end{align}
where in the final step we use (\ref{eqn:applUunitary}). Note that (\ref{eqn:QIdeltas}) is a quadratic inequality in $(\E\del_s^p)^{1/2p}$. By solving it and subsequently squaring both sides we find
\begin{align*}
(\E\del_s^p)^{1/p} & \lesssim K\sqrt{\frac{s}{m}}\log s\sqrt{\log m}\sqrt{\log N} + K^2\frac{s}{m}\log^2 s\log m\log N \\
& \qquad \qquad + \sqrt{p} K\sqrt{\frac{s}{m}} + p K^2\frac{s}{m}.
\end{align*}
Since $1\leq p<\infty$ was arbitrary, Lemma~\ref{lem:MomentsToTails} implies that for any $u\geq 1$,
\begin{align*}
\bP\Big(\del_s & \gtrsim K\sqrt{\frac{s}{m}}\log s\sqrt{\log m}\sqrt{\log N} + K^2\frac{s}{m}\log^2 s\log m\log N \\
& \qquad \qquad \qquad + \sqrt{u}K\sqrt{\frac{s}{m}} + u K^2\frac{s}{m}\Big) \leq e^{-u}.
\end{align*}
Therefore, if we set $u=\log(\eta^{-1})$ and pick $m$ as in (\ref{eqn:conditionFourier}), then
\begin{align*}
\bP(\del_s\geq \del) & \leq \bP\Big(\del_s\gtrsim K\sqrt{\frac{s}{m}}\log s\sqrt{\log m}\sqrt{\log N} + K^2\frac{s}{m}\log^2 s\log m\log N \\
 & \qquad \qquad \qquad + \sqrt{\log(\eta^{-1})}K\sqrt{\frac{s}{m}} + \log(\eta^{-1}) K^2\frac{s}{m}\Big) \leq \eta.
\end{align*}
\end{proof}
A small modification of the proof of Theorem~\ref{thm:RIPFourier} yields the following result. It implies in particular the restricted isometry property of matrices obtained by sampling from bounded orthonormal systems, which was established in \cite[Theorem 8.4]{Rau10} (see also \cite[Theorem 12.32]{FoR13}).
\begin{theorem}
\label{thm:BOS}
Let $A$ be an $m\ti N$ random matrix with rows $\frac{1}{\sqrt{m}}X_1,\ldots,\frac{1}{\sqrt{m}}X_m$. Suppose that
$$\frac{1}{m}\sum_{i=1}^m \E|\langle X_i,x\rangle|^2 = \|x\|_2^2 \qquad \mathrm{for \ all} \ x\in \C^N$$
and
$$\max_{1\leq i\leq m} \|X_i\|_{\infty} \leq K.$$
Then there exist universal constants $d_1,d_2>0$ such that for any given $s \in \N$ and $0<\del,\eta<1$, we have $\bP(\del_s(A) \geq \del) \leq \eta$, provided that (\ref{eqn:conditionFourier}) holds.
\end{theorem}

\section{Supremum of an empirical process}
\label{sec:EmpProc}

In this section we investigate tail bounds for suprema of empirical processes. We begin by applying Theorem~\ref{thm:mixedTailTal} to these processes. For this purpose we recall Bernstein's inequality. For a proof of this result, see for example \cite[Theorem 2.10]{BLM13}.
\begin{lemma}
(Bernstein's inequality) Let $X_1,\ldots,X_m$ be real-valued, independent, mean-zero random variables and suppose that for some constants $\si,K>0$,
$$\frac{1}{m}\sum_{i=1}^m \E|X_i|^q\leq \frac{q!}{2}\si^2 K^{q-2}, \qquad (q=2,3,\ldots).$$
Then,
\begin{equation}
\label{eqn:Bernstein}
\bP\Big(\Big|\frac{1}{m}\sum_{i=1}^m X_i\Big|\geq \frac{\si}{\sqrt{m}}\sqrt{2u} + \frac{K}{m}u\Big)\leq 2\exp(-u) \qquad (u\geq 0).
\end{equation}
In particular, if $X_1,\ldots,X_m$ are subexponential, then
\begin{equation}
\label{eqn:BernsteinPsi1}
\bP\Big(\Big|\frac{1}{m}\sum_{i=1}^m X_i\Big|\geq \frac{\nu}{\sqrt{m}}\sqrt{2u} + \frac{\ka}{m}u\Big)\leq 2\exp(-u) \qquad (u\geq 0),
\end{equation}
where $\nu^2 = \frac{1}{m}\sum_{i=1}^m \|X_i\|_{\psi_1}^2$ and $\ka=\max_{1\leq i\leq m} \|X_i\|_{\psi_1}$.
\end{lemma}
Consider the following setup. Fix an $m\in \N$ and consider $m$ probability spaces $(\Om_1,\bP_1),\ldots,(\Om_m,\bP_m)$. Suppose that we are given a parameter set $T$ consisting of $m$-tuples $t=(t_1,\ldots,t_m)$. For every $t\in T$ we are given an $m$-tuple $X_t = (X_{t_1},\ldots,X_{t_m})$ of subexponential random variables $X_{t_i}:\Om_i\rightarrow\R$. We consider the empirical process
$$E_t = \frac{1}{m} \sum_{i=1}^m X_{t_i} - \E X_{t_i}.$$
In the terminology used here, Bernstein's inequality (\ref{eqn:BernsteinPsi1}) implies that the process $(E_t)_{t\in T}$ has a mixed tail with respect to the metrics $(\frac{1}{m}d_1,\frac{1}{\sqrt{m}}d_2)$, where
\begin{align*}
d_1(s,t) & = \max_{1\leq i\leq m} \|X_{t_i} - X_{s_i}\|_{\psi_1}, \\
d_2(s,t) & = \Big(\frac{1}{m}\sum_{i=1}^m \|X_{t_i} - X_{s_i}\|_{\psi_1}^2\Big)^{1/2}.
\end{align*}
Theorem~\ref{thm:mixedTailTal} can directly be applied to find the following tail bound.
\begin{corollary}
\label{cor:supEmpProc}
Let $E_t$ be as above and let $\si,K>0$ be constants such that
$$\sup_{t\in T}\frac{1}{m}\sum_{i=1}^m \E|X_{t_i} - \E X_{t_i}|^q\leq \frac{q!}{2}\si^2 K^{q-2}, \qquad (q=2,3,\ldots).$$
Then, for any $1\leq p<\infty$,
$$\Big(\E \sup_{t\in T} |E_t|^p\Big)^{1/p} \lesssim \Big(\frac{1}{\sqrt{m}}\ga_2(T,d_2) + \frac{1}{m}\ga_1(T,d_1)\Big) + \sqrt{p}\frac{\si}{\sqrt{m}} + p\frac{K}{m}.$$
In particular, there exist constants $c,C>0$ such that for any $u\geq 1$,
\begin{equation}
\label{eqn:supEmpProcTail}
\bP\Big(\sup_{t\in T}|E_t| \geq C\Big(\frac{1}{\sqrt{m}}\ga_2(T,d_2) + \frac{1}{m}\ga_1(T,d_1)\Big) + c\Big(\frac{\si}{\sqrt{m}}\sqrt{u} + \frac{K}{m}u\Big)\Big)\leq e^{-u}.
\end{equation}
\end{corollary}
Inequality (\ref{eqn:supEmpProcTail}) can be compared to a deviation inequality in \cite[Theorem 8]{GeL12}. In this result the generic chaining estimate
$$\frac{1}{\sqrt{m}}\ga_2(T,d_2) + \frac{1}{m}\ga_1(T,d_1)$$
occurring in (\ref{eqn:supEmpProcTail}) is replaced by an estimate obtained by `chaining along a tree', which is a variation of classical (i.e., non-generic) chaining. As a consequence, the estimate (\ref{eqn:supEmpProcTail}) is in general better. The parameters $\si$ and $K$ governing the tail behavior in (\ref{eqn:supEmpProcTail}) are the same in \cite{GeL12}. Note that \cite{GeL12} also contains a tail bound obtained by `generic chaining along a tree' (see Theorem 3 there). However, the parameters governing the tail behavior in the latter result still depend on the metric complexity of the index set $T$.

\subsection{Supremum of an average of squares}
\label{sec:avSquares}

We continue in the above setup, but now assume that the random variables $X_{t_i}:\Om_i\rightarrow\R$ are \emph{subgaussian} instead of subexponential. For every $t$ in $T$ we consider the average
\begin{equation}
\label{eqn:AtDef}
A_t = \frac{1}{m} \sum_{i=1}^m X_{t_i}^2 - \E X_{t_i}^2.
\end{equation}
Clearly, we can use Corollary~\ref{cor:supEmpProc} to find $L^p$-bounds for $\sup_{t\in T} A_t$. In this section, we will however look for a more natural bound involving a metric defined in terms of the $X_{t_i}$ instead of their squares. The main result, Theorem~\ref{thm:supAverages}, improves a result in this direction of Mendelson, Pajor and Tomczak-Jaegermann \cite{MPT07} (see Theorem~\ref{thm:MPT} and Corollary~\ref{cor:MPTImproved} for a detailed comparison). We consider the metric $d_{\psi_2}$ on $T$ defined by
\begin{equation*}
d_{\psi_2}(s,t) = \max_{i=1,\ldots,m} \|X_{s_i}-X_{t_i}\|_{\psi_2}.
\end{equation*}
We define the associated radius of $T$ by
$$\DelO_{\psi_2}(T) = \sup_{t\in T} \max_{i=1,\ldots,m} \|X_{t_i}\|_{\psi_2}$$
and usually write $\DelO_{\psi_2}$ instead of $\DelO_{\psi_2}(T)$ for brevity. Finally, we denote by $\mu_m$ the normalized counting measure on $\{1,\ldots,m\}$. With this notation,
$$\|X_t-X_s\|_{L^2(\mu_m)} = \Big(\frac{1}{m} \sum_{i=1}^m (X_{t_i}-X_{s_i})^2\Big)^{1/2}.$$
\begin{lemma}
\label{lem:HTail}
Let $s,t\in T$. For any $u\geq 1$,
\begin{equation}
\label{eqn:HSubg}
\bP\Big(\|X_t-X_s\|_{L^2(\mu_m)} \geq u2(1+\sqrt{2})d_{\psi_2}(s,t)\Big) \leq 2\exp(-mu^2).
\end{equation}
\end{lemma}
\begin{proof}
Consider Bernstein's inequality (\ref{eqn:BernsteinPsi1}). If $u\geq m$ then, using that $\nu\leq \ka$, we have
$$\sqrt{2}\nu\sqrt{\frac{u}{m}} \leq \sqrt{2}\ka \frac{u}{m}.$$
Thus, for $u\geq m$,
$$\bP\Big(\Big|\frac{1}{m}\sum_{i=1}^m X_i\Big|\geq (1+\sqrt{2})\ka\frac{u}{m}\Big) \leq 2\exp(-u).$$
We apply this inequality for $X_i = (X_{t_i} - X_{s_i})^2 - \E(X_{t_i} - X_{s_i})^2$. Note that in this case
$$\ka = \max_{1\leq i\leq m} \|(X_{t_i} - X_{s_i})^2 - \E(X_{t_i} - X_{s_i})^2\|_{\psi_1} \leq 2\max_{1\leq i\leq m} \|X_{t_i}-X_{s_i}\|_{\psi_2}^2 = 2d_{\psi_2}^2(s,t).$$ Therefore, we find for any $u\geq 1$,
\begin{align*}
& \bP\Big(\frac{1}{m} \sum_{i=1}^m (X_{t_i} - X_{s_i})^2 - \frac{1}{m} \sum_{i=1}^m \E(X_{t_i} - X_{s_i})^2 \geq u 2(1+\sqrt{2}) d_{\psi_2}^2(s,t)\Big) \\
& \qquad \qquad \qquad \qquad \qquad \qquad \qquad \qquad \leq 2\exp(-mu).
\end{align*}
Since
$$\frac{1}{m} \sum_{i=1}^m \E(X_{t_i} - X_{s_i})^2 \leq \max_{i=1,\ldots,m}\|X_{t_i} - X_{s_i}\|_{\psi_2}^2 = d_{\psi_2}^2(s,t),$$
we deduce that
\begin{align*}
& \bP\Big(\|X_t-X_s\|_{L^2(\mu_m)} \geq u2(1+\sqrt{2})d_{\psi_2}(s,t)\Big) \\
& \qquad = \bP\Big(\frac{1}{m} \sum_{i=1}^m (X_{t_i} - X_{s_i})^2 - \frac{1}{m} \sum_{i=1}^m \E(X_{t_i} - X_{s_i})^2 \\
& \qquad \qquad \qquad \geq u^2 4(1+\sqrt{2})^2d_{\psi_2}^2(s,t) - \frac{1}{m} \sum_{i=1}^m \E(X_{t_i} - X_{s_i})^2\Big) \\
& \qquad \leq \bP\Big(\frac{1}{m} \sum_{i=1}^m (X_{t_i} - X_{s_i})^2 - \frac{1}{m} \sum_{i=1}^m \E(X_{t_i} - X_{s_i})^2 \\
& \qquad \qquad \qquad \geq 2(1+\sqrt{2})\Big(2(1+\sqrt{2})u^2-\frac{1}{2(1+\sqrt{2})}\Big)d_{\psi_2}^2(s,t)\Big) \\
& \qquad \leq 2\exp\Big(-m\Big(2(1+\sqrt{2})u^2-\frac{1}{2(1+\sqrt{2})}\Big)\Big) \leq 2\exp(-m u^2),
\end{align*}
if $u\geq 1$.
\end{proof}
The following two tail bounds will be used in the proof of our main result.
\begin{lemma}
\label{lem:subgSubexPart}
Let $s,t \in T$ and $n\in \N$. If $2^{n/2}\leq \sqrt{m}$ then for any $u\geq 1$,
$$\bP\Big(|A_t-A_s|\geq u2^{n/2}\frac{10\DelO_{\psi_2}}{\sqrt{m}}d_{\psi_2}(s,t)\Big)\leq 2\exp(-2^nu).$$
On the other hand, if $2^{n/2}\geq \sqrt{m}$, then for any $u\geq 1$,
$$\bP\Big(\|X_t-X_s\|_{L^2(\mu_m)} \geq \sqrt{u}2^{n/2}\frac{5}{\sqrt{m}}d_{\psi_2}(s,t)\Big)\leq 2\exp(-2^nu).$$
\end{lemma}
\begin{proof}
Suppose first that $2^{n/2}\leq \sqrt{m}$. We apply Bernstein's inequality (\ref{eqn:BernsteinPsi1}) with $X_i = X_{t_i}^2 - X_{s_i}^2 - \E(X_{t_i}^2 - X_{s_i}^2)$. Let us first estimate the deviation parameters. We use the H\"{o}lder type inequality (\ref{eqn:C-SPsi}) to obtain
\begin{align*}
\nu & = \Big(\frac{1}{m}\sum_{i=1}^m \|X_{t_i}^2 - X_{s_i}^2 - \E(X_{t_i}^2 - X_{s_i}^2)\|_{\psi_1}^2\Big)^{1/2} \\
& \leq 2\max_{1\leq i\leq m}\|X_{t_i}^2 - X_{s_i}^2\|_{\psi_1} \\
& \leq 2\max_{1\leq i\leq m}\|X_{t_i} - X_{s_i}\|_{\psi_2}\|X_{t_i}+X_{s_i}\|_{\psi_2} \\
& \leq 4\DelO_{\psi_2}d_{\psi_2}(s,t).
\end{align*}
Similarly, $\kappa\leq 4\DelO_{\psi_2}d_{\psi_2}(s,t)$. Thus, by (\ref{eqn:BernsteinPsi1}), for any $v\geq 0$,
$$\bP\Big(|A_t-A_s|\geq \sqrt{2v}\frac{4\DelO_{\psi_2}d_{\psi_2}(s,t)}{\sqrt{m}} + v\frac{4\DelO_{\psi_2}d_{\psi_2}(s,t)}{m}\Big)\leq 2\exp(-v).$$
Taking $v=2^nu$ yields
$$\bP\Big(|A_t-A_s|\geq \sqrt{2u}2^{n/2}\frac{4\DelO_{\psi_2}d_{\psi_2}(s,t)}{\sqrt{m}} + u2^n\frac{4\DelO_{\psi_2}d_{\psi_2}(s,t)}{m}\Big)\leq 2\exp(-2^nu).$$
Now observe that $2^{n/2}\leq \sqrt{m}$ implies that
$$2^n\frac{4\DelO_{\psi_2}d_{\psi_2}(s,t)}{m} \leq 2^{n/2}\frac{4\DelO_{\psi_2}d_{\psi_2}(s,t)}{\sqrt{m}}.$$
Therefore,
$$\bP\Big(|A_t-A_s|\geq (\sqrt{2u}+u)2^{n/2}\frac{4\DelO_{\psi_2}d_{\psi_2}(s,t)}{\sqrt{m}}\Big)\leq 2\exp(-2^nu).$$
By using that $u\geq 1$ we obtain the first assertion.\par
Suppose now that $2^{n/2}\geq \sqrt{m}$. Lemma~\ref{lem:HTail} implies that for any $v\geq 1$,
$$\bP\Big(\|X_t-X_s\|_{L^2(\mu_m)} \geq v5d_{\psi_2}(s,t)\Big) \leq 2\exp(-mv^2)$$
Let $v=2^{n/2}\sqrt{u}\frac{1}{\sqrt{m}}$, then $v\geq 1$ and therefore,
\begin{align*}
& \bP\Big(\|X_t-X_s\|_{L^2(\mu_m)} \geq \sqrt{u}2^{n/2}\frac{5}{\sqrt{m}}d_{\psi_2}(s,t)\Big) \\
& \qquad \qquad \leq 2\exp\Big(-m\Big(2^{n/2}\sqrt{u}\frac{1}{\sqrt{m}}\Big)^2\Big) = 2\exp(-2^nu).
\end{align*}
\end{proof}
We are now ready to prove the main result of this section.
\begin{theorem}
\label{thm:supAverages}
Let $(A_t)_{t\in T}$ be the process of averages defined in (\ref{eqn:AtDef}). Let $\si,K$ be constants satisfying
\begin{equation}
\label{eqn:assSupAverages}
\sup_{t\in T}\frac{1}{m}\sum_{i=1}^m \E|X_{t_i}^2 - \E X_{t_i}^2|^q\leq \frac{q!}{2}\si^2K^{q-2} \qquad (q=2,3,\ldots).
\end{equation}
For any $1\leq p<\infty$,
\begin{align*}
\Big(\E\sup_{t\in T} |A_t|^p\Big)^{\frac{1}{p}} & \lesssim \frac{1}{m}\ga_{2,p}^2(T,d_{\psi_2}) + \frac{1}{\sqrt{m}}\DelO_{\psi_2}(T)\ga_{2,p}(T,d_{\psi_2}) + \sqrt{p} \frac{\si}{\sqrt{m}} + p\frac{K}{m}.
\end{align*}
As a consequence, there are constants $c,C>0$ such that for all $u\geq 1$,
$$\bP\Big(\sup_{t\in T} |A_t| \geq C\Big(\frac{1}{m}\ga_2^2(T,d_{\psi_2}) + \frac{\DelO_{\psi_2}(T)}{\sqrt{m}}\ga_2(T,d_{\psi_2})\Big) + c\Big(\sqrt{u}\frac{\si}{\sqrt{m}} + u\frac{K}{m}\Big)\Big) \leq e^{-u}.$$
\end{theorem}
Note that we can always take the parameters
\begin{align*}
\si & = \sup_{t\in T} \Big(\frac{1}{m}\sum_{i=1}^m \|X_{t_i}\|_{\psi_2}^4\Big)^{1/2} \\
K & = \sup_{t\in T} \max_{1\leq i\leq m} \|X_i\|_{\psi_2}^2.
\end{align*}
\begin{proof}
We again write $\DelO_{\psi_2}:=\DelO_{\psi_2}(T)$ for brevity. Set $l=\lfloor\log_2(p)\rfloor$. Let $\cT$ be an optimal admissible sequence for $\ga_{2,p}(T,d_{\psi_2})$ and let $\pi_n(t) = \argmin_{s\in T_n}d_{\psi_2}(s,t)$. We divide $\N_{>l}$ into two disjoint parts given by
\begin{align*}
I_{\mathrm{subg}} = \Big\{n \in \N_{>l} \ : \ 2^{n/2} \leq \sqrt{m}\Big\}, \qquad I_{\mathrm{subex}} = \Big\{n \in \N_{>l} \ : \ 2^{n/2} > \sqrt{m}\Big\}.
\end{align*}
We write the telescoping sum
\begin{equation}
\label{eqn:AtSplit}
A_t = \sum_{n\in I_{\mathrm{subg}}} A_{\pi_n(t)} - A_{\pi_{n-1}(t)} + \sum_{n\in I_{\mathrm{subex}}} A_{\pi_n(t)} - A_{\pi_{n-1}(t)} + A_{\pi_l(t)}.
\end{equation}
By Lemma~\ref{lem:subgSubexPart}, if $n \in I_{\mathrm{subg}}$ then for all $t\in T$ and $u\geq 1$,
$$\bP\Big(|A_{\pi_n(t)} - A_{\pi_{n-1}(t)}| \geq u\frac{10\DelO_{\psi_2}}{\sqrt{m}}2^{n/2}d_{\psi_2}(\pi_n(t),\pi_{n-1}(t))\Big) \leq 2\exp(-2^nu),$$
whereas if $n \in I_{\mathrm{subex}}$, then for all $u\geq 1$,
$$\bP\Big(\|X_{\pi_n(t)} - X_{\pi_{n-1}}(t)\|_{L^2(\mu_m)} \geq \sqrt{u}\frac{5}{\sqrt{m}}2^{n/2}d_{\psi_2}(\pi_n(t),\pi_{n-1}(t))\Big) \leq 2\exp(-2^nu).$$
Let $\Om_{u,p}$ be the event
\begin{align*}
& \forall n \in I_{\mathrm{subg}}, \ t\in T \ : \ |A_{\pi_n(t)} - A_{\pi_{n-1}(t)}| \leq u\frac{10\DelO_{\psi_2}}{\sqrt{m}}2^{n/2}d_{\psi_2}(\pi_n(t),\pi_{n-1}(t)), \\
& \forall n \in I_{\mathrm{subex}}, \ t\in T \ : \ \|X_{\pi_n(t)} - X_{\pi_{n-1}}(t)\|_{L^2(\mu_m)} \leq \sqrt{u}\frac{5}{\sqrt{m}}2^{n/2}d_{\psi_2}(\pi_n(t),\pi_{n-1}(t)).
\end{align*}
Since for any $n>l$ the number of pairs $(\pi_n(t),\pi_{n-1}(t))$ is bounded by $|T_n| \ |T_{n-1}|\leq 2^{2^n}2^{2^{n-1}}\leq 2^{2^{n+1}}$, Lemma~\ref{lem:unionBoundEst} implies that there is an absolute constant $c>0$ such that if $u\geq 2$,
$$\bP(\Om_{u,p}^c) \leq c \exp(-p u/4).$$
If the event $\Om_{u,p}$ occurs, then for any given $t\in T$,
\begin{align*}
\Big|\sum_{n\in I_{\mathrm{subg}}} A_{\pi_n(t)} - A_{\pi_{n-1}(t)}\Big| & \leq u\frac{10\DelO_{\psi_2}}{\sqrt{m}}\sum_{n\in I_{\mathrm{subg}}} 2^{n/2}d_{\psi_2}(\pi_n(t),\pi_{n-1}(t)) \\
& \leq u\frac{10(1+\sqrt{2})\DelO_{\psi_2}}{\sqrt{m}}\ga_{2,p}(T,d_{\psi_2}).
\end{align*}
For the subexponential part we write
\begin{align}
\label{eqn:subexSplit}
& \Big|\sum_{n\in I_{\mathrm{subex}}} A_{\pi_n(t)} - A_{\pi_{n-1}(t)}\Big| \nonumber \\
& \qquad = \Big|\sum_{n\in I_{\mathrm{subex}}} \frac{1}{m} \sum_{i=1}^m X_{\pi_n(t)_i}^2 - X_{\pi_{n-1}(t)_i}^2 - \E\Big(\frac{1}{m} \sum_{i=1}^m X_{\pi_n(t)_i}^2 - X_{\pi_{n-1}(t)_i}^2\Big)\Big|.
\end{align}
By the Cauchy-Schwarz inequality in $L^2(\mu_m)$, we find
\begin{align*}
& \frac{1}{m} \sum_{i=1}^m X_{\pi_n(t)_i}^2 - X_{\pi_{n-1}(t)_i}^2 \\
& \qquad = \frac{1}{m} \sum_{i=1}^m \Big(X_{\pi_n(t)_i} - X_{\pi_{n-1}(t)_i}\Big)\Big(X_{\pi_n(t)_i} + X_{\pi_{n-1}(t)_i}\Big) \\
& \qquad \leq \Big(\frac{1}{m} \sum_{i=1}^m \Big(X_{\pi_n(t)_i} - X_{\pi_{n-1}(t)_i}\Big)^2\Big)^{1/2} \Big(\frac{1}{m} \sum_{i=1}^m \Big(X_{\pi_n(t)_i} + X_{\pi_{n-1}(t)_i}\Big)^2\Big)^{1/2} \\
& \qquad \leq \|X_{\pi_n(t)} - X_{\pi_{n-1}(t)}\|_{L^2(\mu_m)} \Big(\Big(A_{\pi_n(t)} + \frac{1}{m} \sum_{i=1}^m \E X_{\pi_n(t)_i}^2\Big)^{1/2} \\
& \qquad \qquad \qquad \qquad \qquad \qquad \qquad \qquad + \Big(A_{\pi_{n-1}(t)} + \frac{1}{m} \sum_{i=1}^m \E X_{\pi_{n-1}(t)_i}^2\Big)^{1/2}\Big) \\
& \qquad \leq 2\Big(\sup_{t\in T}|A_t| + \DelO_{\psi_2}^2\Big)^{1/2}\sqrt{u}2^{n/2}\frac{5}{\sqrt{m}}d_{\psi_2}(\pi_n(t),\pi_{n-1}(t)) \\
& \qquad \leq \Big(\sup_{t\in T}|A_t|^{1/2} + \DelO_{\psi_2}\Big)\sqrt{u}2^{n/2}\frac{10}{\sqrt{m}}d_{\psi_2}(\pi_n(t),\pi_{n-1}(t)).
\end{align*}
We now apply this estimate in (\ref{eqn:subexSplit}) and find
\begin{align*}
& \Big|\sum_{n\in I_{\mathrm{subex}}} A_{\pi_n(t)} - A_{\pi_{n-1}(t)}\Big| \\
& \ \leq \Big(\sup_{t\in T}|A_t|^{1/2} + \E\Big(\sup_{t\in T}|A_t|^{1/2}\Big) + 2\DelO_{\psi_2}\Big)\sqrt{u}\frac{10}{\sqrt{m}}\sum_{n\in I_{\mathrm{subex}}}2^{n/2}d_{\psi_2}(\pi_n(t),\pi_{n-1}(t)) \\
& \ \leq \Big(\sup_{t\in T}|A_t|^{1/2} + \E\Big(\sup_{t\in T}|A_t|^{1/2}\Big) + 2\DelO_{\psi_2}\Big)\sqrt{u}\frac{10(1+\sqrt{2})}{\sqrt{m}}\ga_{2,p}(T,d_{\psi_2}).
\end{align*}
In conclusion, if $\Om_{u,p}$ occurs then we find using (\ref{eqn:AtSplit})
\begin{align*}
\sup_{t \in T} |A_t| & \leq \sqrt{u}\frac{10+10\sqrt{2}}{\sqrt{m}}\ga_{2,p}(T,d_{\psi_2})\Big(\sup_{t\in T}|A_t|^{1/2} + \E\Big(\sup_{t\in T}|A_t|^{1/2}\Big)\Big) \\
& \qquad + u\frac{(30+30\sqrt{2})\DelO_{\psi_2}}{\sqrt{m}}\ga_{2,p}(T,d_{\psi_2}) + \sup_{t\in T} |A_{\pi_l(t)}|,
\end{align*}
which is a quadratic inequality in $\sup_{t\in T}|A_t|^{1/2}$. By solving this inequality, we obtain
\begin{align*}
\sup_{t\in T}|A_t|^{1/2} & \leq \sqrt{u}\frac{25}{\sqrt{m}}\ga_{2,p}(T,d_{\psi_2}) + \Big(\sqrt{u}\frac{25}{\sqrt{m}}\ga_{2,p}(T,d_{\psi_2})\E\Big(\sup_{t\in T}|A_t|^{1/2}\Big) \\
& \qquad \qquad \qquad + u\frac{75\DelO_{\psi_2}}{\sqrt{m}}\ga_{2,p}(T,d_{\psi_2}) + \sup_{t\in T} |A_{\pi_l(t)}|\Big)^{1/2},
\end{align*}
which implies that
\begin{align*}
&\sup_{t\in T}|A_t|^{1/2} - \sup_{t\in T} |A_{\pi_l(t)}|^{1/2} \\
& \leq \sqrt{u}\frac{25}{\sqrt{m}}\ga_{2,p}(T,d_{\psi_2}) + \sqrt{u}\Big(\Big(25\E\Big(\sup_{t\in T}|A_t|^{1/2}\Big) + 75\DelO_{\psi_2}\Big)\frac{1}{\sqrt{m}}\ga_{2,p}(T,d_{\psi_2})\Big)^{1/2}
\end{align*}
In conclusion, if $u\geq 2$, then
\begin{align*}
& \bP\Big(\sup_{t\in T}|A_t|^{1/2} - \sup_{t\in T} |A_{\pi_l(t)}|^{1/2} \\
& \ \ \ \ \geq \sqrt{u}\Big(\frac{25}{\sqrt{m}}\ga_{2,p}(T,d_{\psi_2}) + \Big(\Big(25\E\Big(\sup_{t\in T}|A_t|^{1/2}\Big)+ 75\DelO_{\psi_2}\Big)\frac{1}{\sqrt{m}}\ga_{2,p}(T,d_{\psi_2})\Big)^{1/2}\Big)\Big) \\
& \ \ \ \ \qquad \qquad \qquad \qquad \qquad \leq c \exp(-pu/4).
\end{align*}
Since the random variable
$$\sup_{t\in T}|A_t|^{1/2} - \sup_{t\in T} |A_{\pi_l(t)}|^{1/2}$$
is clearly positive, we can now apply Lemma~\ref{lem:LpBdElem} (with $\al=2$) to obtain
\begin{align*}
& \Big(\E\Big(\sup_{t\in T}|A_t|^{1/2} - \sup_{t\in T} |A_{\pi_l(t)}|^{1/2}\Big)^p\Big)^{\frac{1}{p}} \\
& \qquad \leq C\Big(\frac{25}{\sqrt{m}}\ga_{2,p}(T,d_{\psi_2}) + \Big(\Big(25\E\Big(\sup_{t\in T}|A_t|^{1/2}\Big) + 75\DelO_{\psi_2}\Big)\frac{1}{\sqrt{m}}\ga_{2,p}(T,d_{\psi_2})\Big)^{1/2}\Big).
\end{align*}
We use the triangle inequality and the trivial bound
$$\E\sup_{t\in T}|A_t|^{1/2}\leq \Big(\E\sup_{t\in T}|A_t|^{p/2}\Big)^{1/p}$$
to write
\begin{align*}
& \Big(\E\sup_{t\in T}|A_t|^{p/2}\Big)^{1/p} \\
& \qquad \leq C\frac{25}{\sqrt{m}}\ga_{2,p}(T,d_{\psi_2}) + C\Big(\frac{75\DelO_{\psi_2}}{\sqrt{m}}\ga_{2,p}(T,d_{\psi_2})\Big)^{1/2} \\
& \qquad \qquad + C\Big(\frac{25}{\sqrt{m}}\ga_{2,p}(T,d_{\psi_2})\Big)^{1/2}\Big(\E\sup_{t\in T}|A_t|^{p/2}\Big)^{1/2p} + \Big(\E\sup_{t\in T} |A_{\pi_l(t)}|^{p/2}\Big)^{1/p}.
\end{align*}
This is a quadratic inequality in $\Big(\E\sup_{t\in T}|A_t|^{p/2}\Big)^{1/2p}$. By solving it and subsequently raising both sides to the fourth power, we arrive at
$$\Big(\E\sup_{t\in T}|A_t|^{p/2}\Big)^{2/p} \lesssim \frac{1}{m}\ga_{2,p}^2(T,d_{\psi_2}) + \frac{\DelO_{\psi_2}}{\sqrt{m}}\ga_{2,p}(T,d_{\psi_2}) + \Big(\E\sup_{t\in T} |A_{\pi_l(t)}|^{p/2}\Big)^{2/p}.$$
Finally, we use Lemma~\ref{lem:supSmallSet}, Bernstein's inequality (\ref{eqn:Bernstein}) and our assumption (\ref{eqn:assSupAverages}) to obtain
\begin{equation*}
\Big(\E\sup_{t\in T} |A_{\pi_l(t)}|^{p/2}\Big)^{2/p} \leq 4\sup_{t\in T}\Big(\E |A_{t}|^{p/2}\Big)^{2/p} \lesssim \sqrt{p} \frac{\si}{\sqrt{m}} + p\frac{K}{m}.
\end{equation*}
This completes the proof.
\end{proof}
Let us now compare Theorem~\ref{thm:supAverages} with \cite[Corollary 1.9]{MPT07}. We consider the following situation. Let $X_1,\ldots,X_m$ be independent copies of a random variable $X:\Om\rightarrow \Theta$, where $\Theta$ is a measurable space. Let $\mu_X$ denote the probability distribution of $X$. Suppose that $\cF$ is a set of real-valued measurable functions on $\Theta$ and consider the process $(Z(f))_{f\in \cF}$ defined by
$$Z(f) = \Big|\frac{1}{m} \sum_{i=1}^m (f^2(X_i) - \E f^2(X_i))\Big|.$$
\begin{theorem}
\label{thm:MPT}
\cite{MPT07} There exist absolute constants $C,c$ such that the following holds. If $\|f\|_{L^2(\mu_X)}=1$ for all $f\in\cF$, then with probability at least
$$1-\exp\Big(-c\min\Big(m,\ga_2^2(\cF,d_{\psi_2})/\DelO_{\psi_2}^2(\cF)\Big)\Big)$$
we have
$$\sup_{f \in \cF} Z(f) \leq C \DelO_{\psi_2}(\cF)\Big(\frac{1}{m}\ga_{2}^2(\cF,d_{\psi_2}) + \frac{1}{\sqrt{m}}\ga_{2}(\cF,d_{\psi_2})\Big).$$
Moreover, if $\cF$ is symmetric, then
$$\E\sup_{f \in \cF} Z(f) \leq C \DelO_{\psi_2}(\cF)\Big(\frac{1}{m}\ga_{2}^2(\cF,d_{\psi_2}) + \frac{1}{\sqrt{m}}\ga_{2}(\cF,d_{\psi_2})\Big).$$
\end{theorem}
Theorem~\ref{thm:supAverages} improves this result in several respects: we can assume the $X_i$ to be only independent instead of i.i.d., we do not need to assume that $\cF$ lies on the $L^2(\mu_X)$-sphere and, most importantly, we get a better deviation inequality.
\begin{corollary}
\label{cor:MPTImproved}
There exist constants $c,C>0$ such that the following holds. Let $X_i:\Om\rightarrow \Theta$, $1\leq i\leq m$ be independent random variables and let $\cF$ be a set of real-valued measurable functions on $\Theta$. Suppose that $\si,K$ are such that
$$\sup_{f\in \cF}\frac{1}{m}\sum_{i=1}^m \E|f^2(X_i)-\E f^2(X_i)|^q\leq \frac{q!}{2}\si^2 K^{q-2} \qquad (q=2,3,\ldots).$$
Then, for any $u\geq 1$,
$$\bP\Big(\sup_{f\in \cF} Z(f)\geq C\Big(\frac{1}{m}\ga_2^2(\cF,d_{\psi_2}) + \frac{\DelO_{\psi_2}(\cF)}{\sqrt{m}}\ga_2(\cF,d_{\psi_2})\Big) + c\Big(\sqrt{u}\frac{\si}{\sqrt{m}} + u\frac{K}{m}\Big)\Big) \leq e^{-u}.$$
\end{corollary}

\section{Supremum of a second order chaos process}
\label{sec:chaos}

In this section we will make use of the Schatten spaces. For any $m\ti n$ matrix $A$ with complex entries $A_{ij}$ we use
$$\|A\|_{S^q} = (\Tr(A^*A)^{q/2})^{1/q} \qquad (1\leq q<\infty), \qquad \|A\|_{S^{\infty}} = \|A\|_{l^2_n\to l^2_m}$$
to denote the Schatten norms of $A$. We use
$$d_q(A_1,A_2) = \|A_1-A_2\|_{S^q}$$
to denote the associated metrics on the $m\ti n$ matrices. Accordingly, for any set $\cA$ of $m\ti n$ matrices and $1\leq q\leq\infty$ we define the radius
$$\DelO_{q}(\cA) = \sup_{A \in \cA} \|A\|_{S^q}.$$
Let $\xi$ be an $n$-dimensional random vector. For any $n\ti n$ matrix $B$ we define the associated \emph{second order chaos} by
$$C_B(\xi) = \xi^* B\xi - \E(\xi^* B\xi) = \sum_{i,j=1}^n B_{ij}(\xi_i\overline{\xi}_j-\E(\xi_i\overline{\xi}_j)).$$
The tail behavior of $C_B(\xi)$ in the case that $\xi$ has subgaussian components was described by Hanson and Wright \cite{HaW71}, see also \cite{RuV13} for a modern proof.
\begin{theorem}
Suppose that $\xi_1,\ldots,\xi_n$ are independent, mean-zero, real-valued random variables and $\max_i\|\xi_i\|_{\psi_2}\leq 1$. Then, there is a universal constant $c>0$ such that for any $u\geq 0$,
\begin{equation}
\label{eqn:H-W}
\bP(|C_B(\xi)|\geq u) \leq 2\exp\Big(-c\min\Big(\frac{u^2}{\|B\|_{S^2}^2},\frac{u}{\|B\|_{S^{\infty}}}\Big)\Big).
\end{equation}
\end{theorem}
In the terminology of Section~\ref{sec:EmpProc}, (\ref{eqn:H-W}) implies that the process $(C_B(\xi))_{B \in \cB}$ has a mixed tail with respect to the pair $(d_{\infty},d_2)$. Thus, by Theorem~\ref{thm:mixedTailTal}
$$\Big(\E\sup_{B\in\cB}|C_B(\xi)|^p\Big)^{1/p} \lesssim \ga_1(\cB,d_{\infty}) + \ga_2(\cB,d_2) + \sqrt{p}\DelO_2(\cB) + p\DelO_{\infty}(\cB).$$
As it turns out, the occurrence of the $\ga_1$-functional in this bound can lead to suboptimal results in certain applications. To mend this, Krahmer, Mendelson and Rauhut proved the following deviation inequality for chaos processes of a special form, which involves only $\ga_2$-functionals.
\begin{theorem}
\cite[Theorem 3.5]{KMR13} Let $\cA$ be a set of $m\ti n$ matrices. Set $\xi=(\xi_1,\ldots,\xi_n)$, where $\xi_1,\ldots,\xi_n$ are independent, mean-zero, unit variance, real-valued, subgaussian random variables. Define
\begin{align*}
E & = \ga_2^2(\cA,d_{\infty}) + \DelO_2(\cA)\ga_2(\cA,d_{\infty}) \\
V & = \DelO_{\infty}(\cA)(\DelO_2(\cA) + \ga_2(\cA,d_{\infty})) \\
U & = \DelO_{\infty}^2(\cA)
\end{align*}
Then, there exist constants $c_1,c_2>0$ depending only on $\|\xi_1\|_{\psi_2},\ldots,\|\xi_n\|_{\psi_2}$ such that for all $u\geq 0$,
$$\bP\Big(\sup_{A\in \cA} \Big|\|A\xi\|_2^2 - \E\|A\xi\|_2^2\Big| \geq c_1E + u \Big) \leq 2\exp\Big(-c_2\min\Big\{\frac{u^2}{V},\frac{u}{U}\Big\}\Big).$$
\end{theorem}
As discussed in \cite{KMR13}, this result has interesting applications in compressed sensing with structured random matrices.\par
Note that due to the appearance of the $\ga_2$-functional in the factor $V$, the bound does not exhibit the correct tail behavior for large $u$. In fact, one would expect from Lemma~\ref{lem:supSmallSet} and the Hanson-Wright bound that $V$ can be replaced by the smaller factor $\DelO_4^2(\cA)$. In Theorem~\ref{thm:chaos} we show that is indeed possible. Our proof follows in general lines the proof of \cite{KMR13}, with some simplifications. For example, we completely avoid the use of the majorizing measures theorem.\par
The main chaining argument in the proof is contained in the following lemma. We follow the proof of \cite[Lemma 3.2]{KMR13}.
\begin{lemma}
\label{lem:chainLpChaosA}
Fix $1\leq p<\infty$. Let $\xi=(\xi_1,\ldots,\xi_n)$ be a random vector with $\max_i\|\xi_i\|_{\psi_2}\leq 1$ and let $\xi'$ be an independent copy of $\xi$ defined on a probability space $(\Om',\cF',\bP')$. Set $l=\lfloor\log_2(p)\rfloor$. Let $\cA$ be a collection of matrices, let $(\cA_n)_{n\geq 0}$ be an optimal admissible sequence for $\ga_{2,p}(\cA,d_{\infty})$ and define an associated sequence of maps $\pi_n:\cA_n\rightarrow\cA$ by $\pi_n(A) = \argmin_{B\in \cA}d_{\infty}(A,B)$. Then,
\begin{equation*}
\Big(\E\sup_{A\in\cA} \Big|\xi^*(A^*A - \pi_l(A)^*\pi_l(A))\xi'\Big|^p\Big)^{1/p} \leq C\ga_{2,p}(\cA,d_{\infty})\Big(\E\sup_{A\in \cA} \|A\xi\|_2^p\Big)^{1/p}.
\end{equation*}
\end{lemma}
\begin{proof}
We make the decomposition
\begin{align*}
& \xi^*(A^*A - \pi_l(A)^*\pi_l(A))\xi' \\
& \qquad = \sum_{n>l} \xi^*\pi_n(A)^*\pi_n(A)\xi' - \xi^*\pi_{n-1}(A)^*\pi_{n-1}(A)\xi' \\
& \qquad = \sum_{n>l} \xi^*(\pi_n(A)-\pi_{n-1}(A))^*\pi_n(A)\xi' + \sum_{n>l} \xi^*\pi_{n-1}(A)^*(\pi_n(A)-\pi_{n-1}(A))\xi' \\
& \qquad =: S_1(A) + S_2(A).
\end{align*}
Let us consider $S_1(A)$. Note that the terms $\xi^*(\pi_n(A)-\pi_{n-1}(A))^*\pi_n(A)\xi'$ are subgaussian in $\xi$ when we condition on $\xi'$. Thus, for any $n>l$,
\begin{align*}
& \bP(|\xi^*(\pi_n(A)-\pi_{n-1}(A))^*\pi_n(A)\xi'|\geq u2^{n/2}\|(\pi_n(A)-\pi_{n-1}(A))^*\pi_n(A)\xi'\|_2) \\
& \qquad \qquad \qquad \qquad \qquad \qquad \leq 2\exp(-u^2 2^n).
\end{align*}
Note that for any $n>l$,
$$|\{((\pi_n(A)-\pi_{n-1}(A)),\pi_n(A)); A\in \cA\}|\leq |\cA_n| \ |\cA_{n-1}| \leq 2^{2^n}2^{2^{n-1}}\leq 2^{2^{n+1}}.$$
Let $\Om_{u,p}$ be the event
\begin{align*}
\forall n>l, \forall A\in \cA \ : &  |\xi^*(\pi_n(A)-\pi_{n-1}(A))^*\pi_n(A)\xi'| \\
& \qquad \qquad \leq u2^{n/2}\|(\pi_n(A)-\pi_{n-1}(A))^*\pi_n(A)\xi'\|_2.
\end{align*}
By Lemma~\ref{lem:unionBoundEst},
$$\bP(\Om_{u,p}^c) \leq c\exp(-p u^2/4) \qquad (u\geq \sqrt{2}).$$
If the event $\Om_{u,p}$ occurs, then
\begin{align*}
|S_1(A)| & \leq \sum_{n>l}u2^{n/2}\|(\pi_n(A)-\pi_{n-1}(A))^*\pi_n(A)\xi'\|_2 \\
& \leq \sum_{n>l}u2^{n/2}\|\pi_n(A)-\pi_{n-1}(A)\|_{S^{\infty}}\|\pi_n(A)\xi'\|_2 \\
& \leq u(1+\sqrt{2})\ga_{2,p}(\cA,d_{\infty})\sup_{A\in\cA} \|A\xi'\|_2.
\end{align*}
In conclusion, for any $u\geq\sqrt{2}$,
$$\bP\Big(\sup_{A\in \cA} |S_1(A)|> u(1+\sqrt{2})\ga_{2,p}(\cA,d_{\infty})\sup_{A\in\cA} \|A\xi'\|_2\Big) \leq c\exp(-p u^2/4).$$
By Lemma~\ref{lem:LpBdElem},
$$\Big(\E\sup_{A\in \cA} |S_1(A)|^p\Big)^{\frac{1}{p}} \leq C\ga_{2,p}(\cA,d_{\infty})\sup_{A\in\cA}\|A\xi'\|_2.$$
Taking the $L^p$-norm over $\Om'$ yields
$$\Big(\E'\E\sup_{A\in \cA} |S_1(A)|^p\Big)^{\frac{1}{p}} \leq C\Big(\E'\sup_{A\in\cA}\|A\xi'\|_2^p\Big)^{\frac{1}{p}}\ga_{2,p}(\cA,d_{\infty}).$$
A very similar argument gives
$$\Big(\E\E'\sup_{A\in \cA} |S_2(A)|^p\Big)^{\frac{1}{p}} \leq C\Big(\E\sup_{A\in\cA}\|A\xi\|_2^p\Big)^{\frac{1}{p}}\ga_{2,p}(\cA,d_{\infty}).$$
The asserted estimate now follows by the triangle inequality.
\end{proof}
In the proof of the main theorem of this section we use the following decoupling inequality due to Arcones and Gin\'{e}.
\begin{lemma}
\label{lem:ArG}
\cite{ArG93} Let $g$ be an $n$-dimensional standard gaussian vector, let $g'$ be an independent copy of $g$ and let $\cB$ be a collection of self-adjoint $n\ti n$ matrices. There is an absolute constant $C>0$ such that for any $1\leq p<\infty$,
$$\Big(\E\sup_{B\in\cB}\Big|g^*Bg - \E(g^*Bg)\Big|^p\Big)^{1/p} \leq C \Big(\E\E'\sup_{B\in\cB}|g^*Bg'|^p\Big)^{1/p}.$$
\end{lemma}
We will also use the following decoupling inequality, which is elementary to prove (see e.g.\ \cite[Theorem 8.11]{FoR13}). Let $\xi_1,\ldots,\xi_n$ be independent, real-valued, mean-zero random variables and let $\xi_1',\ldots,\xi_n'$ be independent copies. Then, for any $1\leq p<\infty$,
\begin{equation}
\label{eqn:decElem}
\Big(\E\sup_{B\in\cB}\Big|\sum_{i\neq j} \xi_i\xi_j B_{ij}\Big|^p\Big)^{1/p} \leq 4\Big(\E\E'\sup_{B\in\cB}|\xi^* B\xi'|^p\Big)^{1/p}
\end{equation}
\begin{theorem}
\label{thm:chaos}
Let $\cA$ be a set of $m\ti n$ matrices. Suppose that $\xi_1,\ldots,\xi_n$ are independent, real-valued, mean-zero random variables, let $\xi=(\xi_1,\ldots,\xi_n)$ and set $\|\xi\|_{\psi_2} = \max_i\|\xi_i\|_{\psi_2}$. For any $1\leq p<\infty$,
\begin{align}
\label{eqn:supChaosImproved}
\Big(\E\sup_{A\in\cA}\Big|\|A\xi\|_2^2 - \E\|A\xi\|_2^2\Big|^p\Big)^{1/p} & \lesssim \|\xi\|_{\psi_2}^2\Big(\ga_{2,p}^2(\cA,d_{\infty}) + \DelO_2(\cA)\ga_{2,p}(\cA,d_{\infty}) \nonumber \\
& \qquad \qquad \qquad \qquad \ \ + \sqrt{p}\DelO_{4}^2(\cA) + p\DelO_{\infty}^2(\cA)\Big).
\end{align}
As a consequence, there are constants $c,C>0$ such that for any $u\geq 1$,
\begin{align*}
\bP\Big(\sup_{A\in\cA}\Big|\|A\xi\|_2^2 - \E\|A\xi\|_2^2\Big| & \geq C\|\xi\|_{\psi_2}^2\Big(\ga_{2}^2(\cA,d_{\infty}) + \DelO_2(\cA)\ga_{2}(\cA,d_{\infty})\Big) \\
& \qquad \qquad \qquad + c\|\xi\|_{\psi_2}^2\Big(\sqrt{u}\DelO_{4}^2(\cA) + u\DelO_{\infty}^2(\cA)\Big)\Big) \leq e^{-u}.
\end{align*}
\end{theorem}
\begin{proof}
By dividing both sides of (\ref{eqn:supChaosImproved}) by $\|\xi\|_{\psi_2}$ if necessary, we may assume that $\|\xi\|_{\psi_2}\leq 1$. Let $l=\lfloor\log_2(p)\rfloor$. Let $(\cA_n)_{n\geq 0}$ and $(\pi_n)_{n\geq 0}$ be as in Lemma~\ref{lem:chainLpChaosA} and write
\begin{align*}
&\sup_{A\in\cA} \xi^*A^*A\xi - \E(\xi^*A^*A\xi) \\
&\qquad \leq \sup_{A\in\cA} \xi^*A^*A\xi - \xi^*\pi_l(A)^*\pi_l(A)\xi - \E(\xi^*A^*A\xi - \xi^*\pi_l(A)^*\pi_l(A)\xi) \\
& \qquad \qquad + \sup_{A\in\cA} \xi^*\pi_l(A)^*\pi_l(A)\xi - \E(\xi^*\pi_l(A)^*\pi_l(A)\xi).
\end{align*}
We continue by estimating the first term. We write $B = B(A,l) := A^*A - \pi_l(A)^*\pi_l(A)$ for brevity. By the triangle inequality and the decoupling inequality (\ref{eqn:decElem}),
\begin{align*}
& \Big(\E\sup_{A\in\cA}|\xi^*B\xi - \E(\xi^*B\xi)|^p\Big)^{1/p} \\
& \qquad \leq \Big(\E\sup_{A\in\cA}\Big|\sum_{i\neq j} \xi_i\xi_j B_{ij}\Big|^p\Big)^{1/p} + \Big(\E\sup_{A\in\cA}\Big|\sum_i (|\xi_i|^2-\E|\xi_i|^2)B_{ii}\Big|^p\Big)^{1/p} \\
& \qquad \leq \Big(\E\E'\sup_{A\in\cA} |\xi^*B\xi'|^p\Big)^{1/p} + \Big(\E\sup_{A\in\cA} \Big|\sum_i (|\xi_i|^2-\E|\xi_i|^2)B_{ii}\Big|^p\Big)^{1/p}.
\end{align*}
Let $\eps$ be a Rademacher vector and let $g$ be a standard Gaussian vector. By symmetrization \cite[Lemma 6.3]{LeT91}, the contraction principle \cite[Lemma 4.6]{LeT91} and de-symmetrization \cite[Lemma 6.3]{LeT91},
\begin{align*}
& \Big(\E\sup_{A\in\cA} \Big|\sum_i (|\xi_i|^2-\E|\xi_i|^2)B_{ii}\Big|^p\Big)^{1/p} \\
& \qquad \leq 2\Big(\E\E_{\eps}\sup_{A\in\cA} \Big|\sum_i \eps_i|\xi_i|^2 B_{ii}\Big|^p\Big)^{1/p} \\
& \qquad \leq 2\Big(\E\E_{\eps}\sup_{A\in\cA} \Big|\sum_i \eps_i g_i^2 B_{ii}\Big|^p\Big)^{1/p} \\
& \qquad \leq 4\Big(\E\sup_{A\in\cA} \Big|\sum_i (g_i^2-1)B_{ii}\Big|^p\Big)^{1/p} + 2 \Big(\E_{\eps}\sup_{A\in\cA} \Big|\sum_i \eps_i B_{ii}\Big|^p\Big)^{1/p} \\
& \qquad \leq 4\Big(\E\sup_{A\in\cA} \Big|g^*Bg-\E(g^*Bg)\Big|^p\Big)^{1/p} + 4\Big(\E\sup_{A\in\cA} \Big|\sum_{i\neq j} g_i g_j B_{ij}\Big|^p\Big)^{1/p} \\
& \qquad \qquad \qquad \qquad + 2 \Big(\E_{\eps}\sup_{A\in\cA} \Big|\sum_i \eps_i B_{ii}\Big|^p\Big)^{1/p} \\
& \qquad \lesssim 32 \Big(\E\E'\sup_{A\in\cA} |g^*Bg'|^p\Big)^{1/p} + 2\Big(\E_{\eps}\sup_{A\in\cA} \Big|\sum_i \eps_i B_{ii}\Big|^p\Big)^{1/p},
\end{align*}
where in the final step we used the inequality decoupling (\ref{eqn:decElem}) and Lemma~\ref{lem:ArG}. We first estimate the second term on the far right hand side. By Khintchine's inequality, for any $C,D \in \cA$,
$$\Big(\E_{\eps}\Big|\sum_i \eps_i (C^*C-D^*D)_{ii}\Big|^p\Big)^{1/p} \leq \sqrt{p} \Big(\sum_i |(C^*C)_{ii}-(D^*D)_{ii}|^2\Big)^{1/2}.$$
Let $C_{(i)}$ denote the $i$-th column of $C$. Then we can estimate
\begin{align*}
& \Big(\sum_i |(C^*C)_{ii}-(D^*D)_{ii}|^2\Big)^{1/2} \\
& \qquad = \Big(\sum_i \Big(\|C_{(i)}\|_2^2 - \|D_{(i)}\|_2^2\Big)^2\Big)^{1/2} \\
& \qquad = \Big(\sum_i \Big(\|C_{(i)}\|_2 - \|D_{(i)}\|_2\Big)^2\Big(\|C_{(i)}\|_2 + \|D_{(i)}\|_2\Big)^2\Big)^{1/2} \\
& \qquad \leq \Big(\sum_i \|C_{(i)} - D_{(i)}\|_2^2(\|C_{(i)}\|_2 + \|D_{(i)}\|_2)^2\Big)^{1/2}.
\end{align*}
Moreover, for any fixed $i$,
\begin{align*}
\|C_{(i)} - D_{(i)}\|_2^2 & = \|(C-D)_{(i)}\|_2^2 \\
& = ((C-D)^*(C-D))_{ii} \leq \|(C-D)^*(C-D)\|_{S^{\infty}} = \|C-D\|_{S^{\infty}}^2.
\end{align*}
In conclusion, we find
\begin{align*}
& \Big(\E_{\eps}\Big|\sum_i \eps_i (C^*C-D^*D)_{ii}\Big|^p\Big)^{1/p} \\
& \qquad \leq 2\sqrt{p} d_{\infty}(C,D) \sup_{A\in\cA} \Big(\sum_i \|A_{(i)}\|_2^2\Big)^{1/2} = 2\sqrt{p}\DelO_2(\cA)d_{\infty}(C,D).
\end{align*}
Thus, by Lemma~\ref{lem:MomentsToTails} the process
$$\Big(\sum_i \eps_i (A^*A)_{ii}\Big)_{A \in\cA}$$
is subgaussian with respect to the metric $\DelO_2(\cA)d_{\infty}$ and Theorem~\ref{thm:chainsplitGeneral} immediately yields
$$\Big(\E_{\eps}\sup_{A\in\cA} \Big|\sum_i \eps_i B_{ii}\Big|^p\Big)^{1/p} \lesssim \DelO_2(\cA)\ga_{2,p}(\cA,d_{\infty}).$$
By Lemma~\ref{lem:chainLpChaosA} and the (quasi-)triangle inequality in $L^{p/2}$,
\begin{align*}
& \Big(\E\E'\sup_{A\in\cA}|\xi^*B\xi'|^p\Big)^{1/p} \\
& \qquad \lesssim \ga_{2,p}(\cA,d_{\infty})\Big(\E\sup_{A \in \cA}\|A\xi\|_2^p\Big)^{1/p} \\
& \qquad \lesssim \ga_{2,p}(\cA,d_{\infty})\Big(\E\Big|\sup_{A \in \cA}\Big|\|A\xi\|_2^2-\E\|A\xi\|_2^2\Big| + \sup_{A\in\cA}\E\|A\xi\|_2^2\Big|^{p/2}\Big)^{1/p} \\
& \qquad \lesssim \ga_{2,p}(\cA,d_{\infty})\Big(\Big(\E\sup_{A \in \cA}\Big|\|A\xi\|_2^2-\E\|A\xi\|_2^2\Big|^p\Big)^{1/2p} + \DelO_2(\cA)\Big).
\end{align*}
Finally, by Lemma~\ref{lem:supSmallSet}, the Hanson-Wright bound (\ref{eqn:H-W}) and Lemma~\ref{lem:TailsToMoments},
\begin{align*}
\Big(\E\sup_{A\in\cA} \Big|\|\pi_l(A)\xi\|_2^2 - \E\|\pi_l(A)\xi\|_2^2\Big|^p\Big)^{1/p} & \leq 2 \sup_{A\in \cA}\Big(\E\Big|\|A\xi\|_2^2 - \E\|A\xi\|_2^2\Big|^p\Big)^{1/p} \\
& \lesssim \sqrt{p}\DelO_{4}^2(\cA) + p\DelO_{\infty}^2(\cA).
\end{align*}
Collecting our estimates, we find
\begin{align*}
& \Big(\E\sup_{A\in\cA}\Big|\|A\xi\|_2^2 - \E\|A\xi\|_2^2\Big|^p\Big)^{1/p} \\
& \qquad \lesssim \ga_{2,p}(\cA,d_{\infty})\Big(\E\sup_{A \in \cA}\Big|\|A\xi\|_2^2-\E\|A\xi\|_2^2\Big|^p\Big)^{1/2p} \\
& \qquad \qquad + \ga_{2,p}(\cA,d_{\infty})\DelO_2(\cA) + \sqrt{p}\DelO_4^2(\cA) + p\DelO_{\infty}^2(\cA).
\end{align*}
By solving this quadratic inequality, we obtain the result.
\end{proof}

\appendix

\section{}

In this appendix we collect some elementary observations that are used throughout the paper. The first lemma states how to pass from moment to tail bounds. The proof is a straightforward consequence of Markov's inequality, see e.g.\ \cite[Propositions 7.11 and 7.15]{FoR13}.
\begin{lemma}
\label{lem:MomentsToTails}
If $X$ is a complex-valued random variable satisfying
$$(\E|X|^p)^{1/p} \leq ap^{1/\al} + b, \qquad \mathrm{for \ all} \ p\geq 1,$$
for some $0<a,\al<\infty$ and $b\geq 0$, then
$$\bP(|X|\geq e^{1/\al}(au + b)) \leq \exp(-u^{\al}/\al) \qquad (u\geq 1).$$
If $X$ satisfies
$$(\E|X|^p)^{1/p} \leq a_1p + a_2 \sqrt{p} + a_3, \qquad \mathrm{for \ all} \ p\geq 1,$$
for some $0\leq a_1,a_2,a_3<\infty$, then
$$\bP(|X|\geq e(a_1u + a_2\sqrt{u} + a_3)) \leq \exp(-u) \qquad (u\geq 1).$$
\end{lemma}
The following observation is a converse statement.
\begin{lemma}
\label{lem:TailsToMoments}
Let $0<\al<\infty$. If a random variable $X$ satisfies
$$\bP(|X|\geq e^{1/\al}au) \leq be^{-u^{\al}/\al} \qquad (u\geq 0),$$
then for any $p\geq 1$,
$$(\E|X|^p)^{1/p} \leq e^{1/2e}a\Big(\sqrt{\frac{2\pi}{\al}}e^{\al/12}b\Big)^{1/p}p^{1/\al}.$$
If $X$ satisfies
\begin{equation}
\label{eqn:TTMmixedTail}
\bP(|X|\geq a_1u + a_2\sqrt{u}) \leq \exp(-u) \qquad (u\geq 0)
\end{equation}
for some $0\leq a_1,a_2<\infty$, then for all $p\geq 1$
$$(\E|X|^p)^{1/p} \leq a_1 2e^{1/(2e)}(\sqrt{2\pi}e^{1/(12p)})^{1/p}e^{-1}p + a_2 2(2e)^{-1/2}e^{1/(2e)}(\sqrt{\pi}e^{1/(6p)})^{1/p} \sqrt{p}.$$
\end{lemma}
\begin{proof}
For a proof of the first statement, see \cite[Proposition 7.13]{FoR13}. For the second assertion, note that (\ref{eqn:TTMmixedTail}) implies
$$
\bP(\tfrac{1}{2}|X|\geq u) \leq \left\{
  \begin{array}{ll}
    e^{-u^2/a_2^2}, & \mathrm{if} \ 0\leq u\leq a_2^2/a_1; \\
    e^{-u/a_1}, & \mathrm{if} \ u\geq a_2^2/a_1.
  \end{array}
\right.
$$
Using integration by parts and a change of variable we find
\begin{align*}
2^{-p}\E|X|^p & = p\int_0^{\infty}u^{p-1}\bP(\tfrac{1}{2}|X|\geq u) \ du \\
& \leq p\int_0^{a_2^2/a_1} u^{p-1}e^{-u^2/a_2^2} \ du + p\int_{a_2^2/a_1}^{\infty} u^{p-1}e^{-u/a_1} \ du \\
& = \tfrac{1}{2}pa_2^p\int_0^{a_2^2/a_1^2} v^{\frac{p}{2}-1} e^{-v} \ dv + pa_1^p \int_{a_2^2/a_1^2}^{\infty} v^{p-1} e^{-v} \ dv \\
& \leq \tfrac{1}{2}pa_2^p \Gamma(p/2) + pa_1^p \Gamma(p),
\end{align*}
where $\Gamma(p)=\int_0^{\infty} v^{p-1}e^{-v} \ dv$ is the gamma function. The result now readily follows using Stirling's formula, which states that
\begin{equation*}
\label{eqn:Stirling}
\Gamma(p) = \sqrt{2\pi} p^{p-1/2} e^{-p} e^{\theta(p)/12p}
\end{equation*}
for some $0\leq \theta(p)\leq 1$.
Indeed,
\begin{equation*}
p\Gamma(p) \leq p^p \sqrt{2\pi}\sqrt{p} e^{-p}e^{1/(12p)}
\end{equation*}
and therefore
$$(p\Gamma(p))^{1/p} \leq p(\sqrt{2\pi}e^{1/(12p)})^{1/p}e^{-1} e^{1/(2e)},$$
where we used that $p^{1/(2p)}\leq e^{1/(2e)}$ if $p\geq 1$. In the same way,
$$(\tfrac{1}{2}p\Gamma(p/2))^{1/p} \leq (2e)^{-1/2}e^{1/(2e)} (\sqrt{\pi}e^{1/(6p)})^{1/p} \sqrt{p}.$$
\end{proof}
The following three lemmas are used in every chaining argument in this paper.
\begin{lemma}
\label{lem:supSmallSet} Fix $1\leq p<\infty$, set $l=\lfloor\log_2(p)\rfloor$ and let $(X_t)_{t\in T}$ be a collection of complex-valued random variables. If $|T|\leq 2^{2^l}$, then
$$\Big(\E\sup_{t\in T}|X_t|^p\Big)^{1/p} \leq 2\sup_{t\in T} (\E|X_t|^p)^{1/p}.$$
\end{lemma}
\begin{proof}
Since $|T|\leq 2^p$,
$$\E\sup_{t\in T} |X_t|^p \leq \sum_{t\in T} \E|X_t|^p \leq |T| \sup_{t\in T} \E|X_t|^p \leq 2^p \sup_{t\in T} \E|X_t|^p.$$
\end{proof}
\begin{lemma}
\label{lem:unionBoundEst}
Fix $1\leq p<\infty$, $0<\al<\infty$, $u\geq 2^{1/\al}$ and set $l=\lfloor\log_2(p)\rfloor$. For every $n>l$ let $(\Om_i^{(n)})_{i\in I_n}$ be a collection of events satisfying
$$\bP(\Om_i^{(n)})\leq 2\exp(-2^nu^{\al}), \qquad \mathrm{for \ all} \ i\in I_n.$$
If $|I_n|\leq 2^{2^{n+1}}$, then for an absolute constant $c\leq 16$,
\begin{equation}
\label{eqn:unionBoundEst}
\bP\Big(\bigcup_{n>l}\bigcup_{i\in I_n} \Om_i^{(n)}\Big) \leq c\exp(-pu^{\al}/4).
\end{equation}
\end{lemma}
\begin{proof}
By a union bound, using that $u^{\al}\geq 2$,
\begin{align*}
\bP\Big(\bigcup_{n>l}\bigcup_{i\in I_n} \Om_i^{(n)}\Big) & \leq \sum_{n>l} 2^{2^{n+1}} 2\exp(-u^{\al} 2^n) \\
& = 2\sum_{n>l} \exp(2(\log 2)2^{n}) \exp(-u^{\al} 2^n) \\
& \leq 2\sum_{n>l} \exp((\log 2-1)u^{\al} 2^{n}).
\end{align*}
Clearly,
\begin{align*}
\sum_{n>l} \exp((\log 2-1)u^{\al} 2^{n}) & = \exp(-2^l u^{\al}/2)\sum_{n>l} \exp((\log 2-1)u^{\al} 2^{n} + 2^l u^{\al}/2) \\
& \leq \exp(-2^l u^{\al}/2)\sum_{n\geq 0} \exp((\log 2-1)u^{\al} 2^{n} + 2^n u^{\al}/4).
\end{align*}
Since $\log 2 - \frac{3}{4}<0$ and $-2^{l}\leq -\frac{p}{2}$, we conclude that (\ref{eqn:unionBoundEst}) holds. Note that
\begin{align*}
c & \leq 2\sum_{n\geq 0}\exp(2^n(2(\log 2 - 1) + \tfrac{1}{2})) \\
& \leq 2 \sum_{n\geq 1}\exp(n(2(\log 2 - 1) + \tfrac{1}{2}))\leq \frac{2}{1-\exp(2(\log 2 - 1) + \tfrac{1}{2})} - 2 \leq 16.
\end{align*}
\end{proof}
\begin{lemma}
\label{lem:LpBdElem}
Fix $1\leq p<\infty$ and $0<\al<\infty$. Let $\ga\geq 0$ and suppose that $\xi$ is a positive random variable such that for some $c,u_*>0$,
$$\bP(\xi> \ga u) \leq c\exp(-p u^{\al}/4) \qquad (u\geq u_*).$$
Then, for a constant $\tilde{c}_{\al}>0$ depending only on $\al$,
$$(\E\xi^p)^{1/p} \leq \ga(\tilde{c}_{\al}c+u_{*}).$$
\end{lemma}
\begin{proof}
By integration by parts and a change of variable,
\begin{align*}
\E\xi^p & = \int_0^{\infty} pu^{p-1} \bP(\xi>u) du \\
& = \ga^p \int_0^{\infty} pv^{p-1} \bP(\xi>v\ga) dv \\
& \leq \ga^p \Big(\int_{u_*}^{\infty} p v^{p-1} c\exp(-p v^{\al}/4) dv + \int_0^{u_*} pv^{p-1} dv\Big) \\
& = \ga^p \Big(c\int_{u_*}^{\infty} p v^{p-1} \exp(-p v^{\al}/4) dv + u_*^p\Big).
\end{align*}
To complete the proof, observe that by another change of variable
\begin{align*}
\int_0^{\infty} p v^{p-1} e^{-pv^{\al}/4} dv & = p^{-p/\al} 2^{p/\al}\frac{2p}{\al} \int_0^{\infty} u^{\frac{2p}{\al}-1} e^{-u^{2}/2} du \\
& = p^{-p/\al}2^{p/\al}\frac{2p}{\al} \frac{\sqrt{2\pi}}{2} \E|g|^{\frac{2p}{\al}-1},
\end{align*}
where $g$ is a standard Gaussian. Since
$$\E|g|^{\frac{2p}{\al}-1} \leq \Big(\frac{2p}{\al}-1\Big)^{\frac{p}{\al}-\frac{1}{2}},$$
we conclude that
$$\int_0^{\infty} p v^{p-1} e^{-pv^{\al}/4} dv \leq \frac{\sqrt{2\pi}}{2}2^{p/\al}\Big(\frac{2}{\al}\Big)^{\frac{p}{\al}+\frac{1}{2}}p^{1/2}.$$
The result follows by combining these estimates.
\end{proof}

\section*{Acknowledgement}

It is a pleasure to thank Holger Rauhut for valuable comments.


\begin{thebibliography}{10}

\bibitem{Ada08}
R.~Adamczak.
\newblock A tail inequality for suprema of unbounded empirical processes with
  applications to {M}arkov chains.
\newblock {\em Electron. J. Probab.}, 13:no. 34, 1000--1034, 2008.

\bibitem{ArG93}
M.~Arcones and E.~Gin{\'e}.
\newblock On decoupling, series expansions, and tail behavior of chaos
  processes.
\newblock {\em J. Theoret. Probab.}, 6(1):101--122, 1993.

\bibitem{BLM13}
S.~Boucheron, G.~Lugosi, and P.~Massart.
\newblock {\em Concentration Inequalities: A Nonasymptotic Theory of
  Independence}.
\newblock Oxford University Press, 2013.

\bibitem{CaT06}
E.~Cand{\`{e}}s and T.~Tao.
\newblock Near-optimal signal recovery from random projections: universal
  encoding strategies?
\newblock {\em IEEE Trans. Inform. Theory}, 52(12):5406--5425, 2006.

\bibitem{Dir14}
S.~Dirksen.
\newblock Dimensionality reduction with subgaussian matrices: a unified theory.
\newblock ArXiv:1402.3973.

\bibitem{FoR13}
S.~Foucart and H.~Rauhut.
\newblock {\em A Mathematical Introduction to Compressive Sensing}.
\newblock Birkha{\"{u}}ser, Boston, 2013.

\bibitem{GeL12}
S.~van~de Geer and J.~Lederer.
\newblock The {B}ernstein--{O}rlicz norm and deviation inequalities.
\newblock {\em To appear in Prob. Theory and Rel. Fields}.
\newblock Arxiv:1111.2450.

\bibitem{HaW71}
D.~Hanson and F.~Wright.
\newblock A bound on tail probabilities for quadratic forms in independent
  random variables.
\newblock {\em Ann. Math. Statist.}, 42:1079--1083, 1971.

\bibitem{KlM05}
B.~Klartag and S.~Mendelson.
\newblock Empirical processes and random projections.
\newblock {\em J. Funct. Anal.}, 225(1):229--245, 2005.

\bibitem{Kol11}
V.~Koltchinskii.
\newblock {\em Oracle {I}nequalities in {E}mpirical {R}isk {M}inimization and
  {S}parse {R}ecovery {P}roblems}.
\newblock Springer, Berlin, 2011.

\bibitem{KMR13}
F.~Krahmer, S.~Mendelson, and H.~Rauhut.
\newblock Suprema of chaos processes and the restricted isometry property.
\newblock {\em To appear in Comm. Pure Appl. Math.}
\newblock ArXiv:1207.0235.

\bibitem{KrR61}
M.~Krasnosel'ski{\u\i} and Ja. Ruticki{\u\i}.
\newblock {\em Convex functions and {O}rlicz spaces}.
\newblock P. Noordhoff Ltd., Groningen, 1961.

\bibitem{Lat11}
R.~Lata{\l}a.
\newblock Weak and strong moments of random vectors.
\newblock In {\em Marcinkiewicz centenary volume}, volume~95 of {\em Banach
  Center Publ.}, pages 115--121. Polish Acad. Sci. Inst. Math., Warsaw, 2011.

\bibitem{Led01}
M.~Ledoux.
\newblock {\em The concentration of measure phenomenon}.
\newblock American Mathematical Society, Providence, RI, 2001.

\bibitem{LeT91}
M.~Ledoux and M.~Talagrand.
\newblock {\em Probability in {B}anach spaces}.
\newblock Springer-Verlag, Berlin, 1991.

\bibitem{MPT07}
S.~Mendelson, A.~Pajor, and N.~Tomczak-Jaegermann.
\newblock Reconstruction and subgaussian operators in asymptotic geometric
  analysis.
\newblock {\em Geom. Funct. Anal.}, 17(4):1248--1282, 2007.

\bibitem{Rau10}
H.~Rauhut.
\newblock Compressive sensing and structured random matrices.
\newblock In {\em Theoretical foundations and numerical methods for sparse
  recovery}, volume~9 of {\em Radon Ser. Comput. Appl. Math.}, pages 1--92.
  Walter de Gruyter, Berlin, 2010.

\bibitem{RuV13}
M.~Rudelson and R.~Vershynin.
\newblock Hanson-{W}right inequality and sub-gaussian concentration.
\newblock ArXiv:1306.2872.

\bibitem{RuV08}
M.~Rudelson and R.~Vershynin.
\newblock On sparse reconstruction from {F}ourier and {G}aussian measurements.
\newblock {\em Comm. Pure Appl. Math.}, 61(8):1025--1045, 2008.

\bibitem{Tal14}
M.~Talagrand.
\newblock {\em Upper and Lower Bounds for Stochastic Processes}.
\newblock Springer, Berlin.
\newblock To appear.

\bibitem{Tal87}
M.~Talagrand.
\newblock Regularity of {G}aussian processes.
\newblock {\em Acta Math.}, 159(1-2):99--149, 1987.

\bibitem{Tal01}
M.~Talagrand.
\newblock Majorizing measures without measures.
\newblock {\em Ann. Probab.}, 29(1):411--417, 2001.

\bibitem{Tal05}
M.~Talagrand.
\newblock {\em The generic chaining}.
\newblock Springer-Verlag, Berlin, 2005.

\bibitem{VaW96}
A.~van~der Vaart and J.~Wellner.
\newblock {\em Weak convergence and empirical processes}.
\newblock Springer-Verlag, New York, 1996.

\bibitem{ViV07}
F.~Viens and A.~Vizcarra.
\newblock Supremum concentration inequality and modulus of continuity for
  sub-{$n$}th chaos processes.
\newblock {\em J. Funct. Anal.}, 248(1):1--26, 2007.

\end{thebibliography}
\end{document}